\documentclass[11pt]{amsart}
\usepackage{amsmath,amsthm,amssymb,enumerate,mathscinet,mathtools}
\usepackage{fullpage}
\usepackage{graphicx}
\usepackage{verbatim}
\usepackage{mathrsfs}
\usepackage{hyperref}
\newcommand{\msc}[1]{\begin{center}MSC2000: #1.\end{center}}
\newtheorem{theorem}{Theorem}[section]
\newtheorem{fact}[theorem]{Fact}
\newtheorem{lemma}[theorem]{Lemma}

\newtheorem{corollary}[theorem]{Corollary}
\newtheorem{conjecture}[theorem]{Conjecture}
\newtheorem{claim}[theorem]{Claim}
\newtheorem{definition}[theorem]{Definition}

\newtheorem{thm}[theorem]{Theorem}
\newtheorem{prop}[theorem]{Proposition}
\newtheorem{obs}[theorem]{Observation}

\newtheorem{cor}[theorem]{Corollary}

\newtheorem{defin}[theorem]{Definition}

\numberwithin{equation}{section}

\def\eps{\varepsilon}

\def\eps{\varepsilon}
\def\card#1{\left| #1 \right|}
\def\vSet#1#2{V_{#1}^{(#2)}}

\def\mc#1{\mathcal{#1}}

\DeclareMathOperator{\Aut}{Aut}
\DeclareMathOperator{\Var}{Var}

\DeclarePairedDelimiter\floor{\lfloor}{\rfloor} 
\DeclarePairedDelimiter\ceil{\lceil}{\rceil}  

\def\COMMENT#1{}

\allowdisplaybreaks

\title{Ramsey properties of randomly perturbed graphs: \\ cliques and cycles}

\author{Shagnik Das and
 Andrew Treglown}

\thanks{
SD: Freie Universit\"at Berlin, Germany, {\tt shagnik@mi.fu-berlin.de}.  Research supported by GIF grant G-1347-304.6/2016. \\
\indent AT: University of Birmingham, United Kingdom, {\tt a.c.treglown@bham.ac.uk}. Research supported by EPSRC grant EP/M016641/1.}

\begin{document}

\begin{abstract}
Given graphs $H_1,H_2$, a graph $G$ is $(H_1,H_2)$-Ramsey if  for every colouring  of the edges of $G$ with red and blue, there is a red copy of $H_1$ or a blue copy of $H_2$.
 In this paper we investigate Ramsey questions in the setting of \emph{randomly perturbed graphs}:
this is a random graph model introduced by Bohman, Frieze and Martin~\cite{bfm1} in
which one starts with a dense graph and then adds a given number of random edges to it.
The study of Ramsey properties of randomly perturbed graphs was initiated by Krivelevich, Sudakov and Tetali~\cite{kst} in 2006; they determined how many random edges must be added to a dense graph  to 
ensure the resulting graph is with high probability $(K_3,K_t)$-Ramsey (for $t\geq 3$).
They also raised the question of generalising this result to pairs of graphs other than  $(K_3,K_t)$.
We make significant progress on this question, giving a precise solution  in the case when  $H_1=K_s$ and $H_2=K_t$ where $s,t \geq 5$.
Although we again show that one requires polynomially fewer edges than in the purely random graph, our result shows that the problem in this case is quite different to the $(K_3,K_t)$-Ramsey question.
Moreover, we give bounds for the corresponding $(K_4,K_t)$-Ramsey question; together with a construction of Powierski~\cite{power} this resolves the $(K_4,K_4)$-Ramsey problem.

We also give a precise solution to the analogous question in the case when both $H_1=C_s$ and $H_2=C_t$ are cycles. Additionally we consider the corresponding multicolour problem.
Our final result gives another generalisation of the Krivelevich, Sudakov and Tetali~\cite{kst} result. Specifically, we determine how many random edges must be added to a dense graph  to 
ensure the resulting graph is with high probability $(C_s,K_t)$-Ramsey (for odd $s\geq 5$ and $t\geq 4$).

To prove our results we combine a mixture of approaches, employing the container method, the regularity method as well as dependent random choice, and apply robust extensions of recent asymmetric random Ramsey results.
\end{abstract}

\date{\today}

\maketitle
\msc{5C55, 5C80, 5D10}

\section{Introduction}

\begin{center}
``Introduce a little anarchy, upset the established order, and everything becomes chaos."
\end{center}
\begin{flushright}
--- The Joker, \emph{The Dark Knight} (2008)
\end{flushright}

Let $G$ and $H$ be graphs and $r \in \mathbb N$. 
We say that $G$ is \emph{$(H,r)$-Ramsey} if every $r$-colouring of $G$ yields a monochromatic copy of $H$ in $G$. 
More generally, for graphs $H_1,\ldots,H_r$, a graph $G$ is \emph{$(H_1,\ldots,H_r)$-Ramsey} if  for any $r$-colouring  of $G$ there is a copy of $H_i$ in colour $i$ for some   $i \in [r]$.
In the case $r = 2$, we shall take the first colour to be red and the second colour to be blue.
Ramsey's classic theorem tells us that if $n\in \mathbb N$ is sufficiently large then $K_n$ is $(H_1,\ldots,H_r)$-Ramsey.
Whilst in general for given graphs $H_1,\dots,H_r$ it seems out of reach to determine the \emph{smallest} $n =: R(H_1,\dots,H_r)$ with this property, much effort has gone into establishing good upper and lower bounds on $R(H_1,\dots,H_r)$
(particularly in the case when the $H_i$ are cliques; see e.g.~\cite{dc, spencer}).

\subsection{Ramsey properties of random graphs}
There has also been significant interest in understanding  Ramsey properties of the \emph{random graph $G(n,p)$}. Recall that $G(n,p)$ has vertex set $[n]:=\{1,\dots, n\}$ and each edge is present with probability $p$, independently of all other choices.
Seminal work of R\"odl and Ruci\'nski~\cite{random1, random2, random3} determined the \emph{threshold} for the   $(H,r)$-Ramsey property in $G(n,p)$ for all fixed $H$ and $r\geq 2$. Before we state their result (in a slightly restricted form) we require two definitions.
Given a graph $H$, set $d_2(H):=0$ if $e(H)=0$; $d_2(H):=1/2$ when $H$ is precisely an edge and define $d_2(H) := (e(H)-1)/(v(H)-2)$ otherwise.
Then define $m_2(H) := \max_{H' \subseteq H}d_2(H')$ to be the \emph{$2$-density} of $H$.
We say that $H$ is \emph{strictly $2$-balanced} if $m_2(H')<m_2(H)$ for all $H' \subsetneq H$.
R\"odl and Ruci\'nski showed that the $2$-density of $H$ is the parameter that governs the threshold for the   $(H,r)$-Ramsey property in $G(n,p)$.

\begin{thm}[\cite{random1, random2, random3}]\label{randomramsey} Let $r \geq 2$ be a positive integer and let $H$ be a graph that is not a forest consisting of stars and paths of length $3$.
There are positive constants $c,C$ such that
$$\lim _{n \rightarrow \infty} \mathbb P [G(n,p) \text{ is } (H,r)\text{-Ramsey}]=\begin{cases}
0 &\text{ if } p<cn^{-1/m_2(H)}; \\
 1 &\text{ if } p>Cn^{-1/m_2(H)}.\end{cases}$$
\end{thm} 

We remark that a short proof of Theorem~\ref{randomramsey} was recently given in \cite{ns}.
There has also been attention on the more general problem of determining the threshold of the $(H_1,\ldots,H_r)$-Ramsey property in $G(n,p)$.
In particular, the focal question in the area is the following conjecture of 
Kohayakawa and Kreuter~\cite{kreu}.
To state it, we need to introduce the \emph{asymmetric density} of $H_1,H_2$ where $m_2(H_1)\geq m_2(H_2)$ via
\begin{equation}\label{2density}
m_2(H_1,H_2):=\max \left \{\frac{e(H_1')}{v(H_1')-2+1/m_2(H_2)}: H_1' \subseteq H_1 \text{ and } e(H_1') \geq 1 \right \}.
\end{equation}
We say that $H_1$ is \emph{strictly balanced with respect to $m_2(\cdot, H_2)$} if no $H'_1 \subsetneq H_1$ with at least one edge maximises (\ref{2density}).
\begin{conjecture}[\cite{kreu}]\label{conjkreu}
For any graphs $H_1,\ldots,H_r$ with $m_2(H_1) \geq \ldots \geq m_2(H_r)>1$, there are positive constants $c,C > 0$ such that
$$
\lim_{n \rightarrow \infty} \mathbb{P}\left[ G(n,p) \text{ is } (H_1,\ldots,H_r)\text{-Ramsey} \right] =\begin{cases}
0 &\text{ if } p<cn^{-1/m_2(H_1,H_2)}; \\
 1 &\text{ if } p>Cn^{-1/m_2(H_1,H_2)}.\end{cases}$$
\end{conjecture}

Note that this conjectured threshold only depends on the `joint density' of the densest two graphs $H_1,H_2$.
Further, notice $m_2(H_1) \geq m_2(H_1,H_2) \geq m_2(H_2)$ with equality if and only if $m_2(H_1)=m_2(H_2)$.
Thus Conjecture~\ref{conjkreu} would generalise Theorem~\ref{randomramsey}.
The initial work on Conjecture~\ref{conjkreu} focused on the cases of cycles and cliques. (We take a similar approach in this paper when considering the analogous question in the perturbed setting.)
In particular, Kohayakawa and Kreuter~\cite{kreu} confirmed Conjecture~\ref{conjkreu} when the $H_i$ are cycles.  When each $H_i$ is a clique, the 0-statement was resolved by Marciniszyn, Skokan, Sp\"ohel and Steger~\cite{msss}, who also observed that the approach used by Kohayakawa and Kreuter~\cite{kreu}  implies the 1-statement of Conjecture~\ref{conjkreu} holds when $H_1$ is strictly 2-balanced \emph{provided} the so-called K\L R conjecture holds.
This latter conjecture was proven by Balogh, Morris and Samotij~\cite{bms}, thereby proving the 1-statement of Conjecture~\ref{conjkreu} holds for strictly 2-balanced graphs $H_1$. Hancock, Staden and Treglown~\cite{hst} then proved 
a general result which implies (a resilient version of)
the 1-statement in 
the case when $m_2(H_1)=m_2(H_2)$.
Very recently, Mousset, Nenadov and Samotij~\cite{mns} have shown the 1-statement is true without any assumptions regarding the balancedness of $H_1$.

\subsection{Ramsey properties of randomly perturbed graphs}
 Note that the results discussed above give us precise information about the Ramsey properties of \emph{typical} graphs of a given density. Indeed, Theorem~\ref{randomramsey} implies that a typical graph of density
$p=\omega( n^{-1/m_2(H)})$ is $(H,r)$-Ramsey whilst a typical graph of density $p=o( n^{-1/m_2(H)})$ is not $(H,r)$-Ramsey.
In this paper we study the question of \emph{how far away} a dense graph is from satisfying a given Ramsey property. 
The model of \emph{randomly
perturbed graphs}, introduced by Bohman, Frieze and Martin~\cite{bfm1}, provides a framework for studying
such questions. In their model one starts with a dense graph and then adds a fixed number of random
edges to it. A natural problem in this setting is to determine how many random edges are required to
ensure that the resulting graph w.h.p. satisfies a given property. For example
Bohman, Frieze and Martin~\cite{bfm1} proved that, given any $n$-vertex graph $G$ of linear minimum degree,
if one adds a linear number of random edges to $G$ then, w.h.p., the resulting graph is Hamiltonian.
In recent years, a whole host of results  have been obtained concerning embedding spanning subgraphs into a randomly perturbed graph, as well as other  properties of the model; see e.g.~\cite{ bwt2, benn, bfkm, bhkmpp, bmpp2, dudek, joos2, kks1, kst, nt}.
The model has also been investigated in the setting of directed graphs and hypergraphs (see e.g.~\cite{bhkm, hanzhao, kks2, mm}). Further, very recently an analogous model of randomly perturbed sets of integers has been studied~\cite{elad}.

The study of Ramsey properties of randomly perturbed graphs was initiated by Krivelevich, Sudakov and Tetali~\cite{kst} in 2006. They considered the question of how many random edges one needs to add to any dense  graph to ensure with high probability
the resulting graph is $(H_1,H_2)$-Ramsey. Specifically, they resolved this question in the case when $H_1=K_t$ and $H_2=K_3$ for any $t \geq 3$. 
In this paper we give several generalisations of their result. Further, we look at the following more refined question: given any fixed $0<d<1$, how many random edges does one need to add to any    graph $G$ of density at least $d$ to ensure with high probability
the resulting graph is $(H_1,H_2)$-Ramsey?
In order to present our results  we define a threshold function below.

\begin{defin}[Perturbed Ramsey threshold probability]
Given a density $0 < d < 1$, a number of colours $r \in \mathbb{N}$ and a sequence of graphs $(H_1, H_2, \hdots, H_r)$, the \emph{perturbed Ramsey threshold probability} $p(n; H_1, H_2, \hdots, H_r, d)$ satisfies the following.
\begin{itemize}
	\item[(i)] If $p = p(n) =\omega( p(n; H_1, H_2, \hdots, H_r, d))$, then for any sequence $(G_n)_{n \in \mathbb{N}}$ of $n$-vertex graphs with density at least $d$, the graph $G_n \cup G(n,p)$ is $(H_1, H_2, \hdots, H_r)$-Ramsey with high probability.
	\item[(ii)] If $p = p(n) =o( p(n; H_1, H_2, \hdots, H_r, d))$, for some sequence $(G_n)_{n \in \mathbb{N}}$ of $n$-vertex graphs with density at least $d$, the graph $G_n \cup G(n,p)$ is with high probability not $(H_1, H_2, \hdots, H_r)$-Ramsey.
\end{itemize}
If it is the case that every sufficiently large graph of density at least $d$ is $(H_1, H_2, \hdots, H_r)$-Ramsey then we define $p(n;H_1, H_2, \hdots, H_r, d):=0$.  In the symmetric case, where $H_1 = H_2 = \hdots = H_r = H$, we denote the threshold by $p(n; r,H, d)$.
\end{defin}

We begin by observing some simple lower and upper bounds on the perturbed Ramsey threshold probability, which will serve as points of reference for our results.

\begin{obs}\label{obs:trivial} The following bounds on $p(n;H_1, H_2, d)$ hold:
\begin{itemize}
	\item[(i)] If $d \le 1 - 1/k$ and $H_2$ is not $k$-partite, then we may take $G$ to be a complete balanced $k$-partite graph and colour all its edges blue.  As long as $G(n,p)$ is $H_1$-free, we may colour all its uncoloured edges red, and so $p(n;H_1, H_2, d)$ is at least the threshold probability for the appearance of $H_1$ in $G(n,p)$.
	\item[(ii)] If $G(n,p)$ alone is already $(H_1, H_2)$-Ramsey, then $G \cup G(n,p)$ will be as well.  Since the 1-statement of Conjecture~\ref{conjkreu} is known to hold, it follows that for any $d$ and graphs $H_1$ and $H_2$ with $m_2(H_1) \ge m_2(H_2) > 1$, we have $p(n; H_1, H_2, d) \le n^{-1/m_2(H_1, H_2)}$.
	\item[(iii)] Suppose there is a $(k+1)$-chromatic graph $H$ that is $(H_1, H_2)$-Ramsey.  If $d > 1 - 1/k$, the Erd\H{o}s--Stone--Simonovits Theorem~\cite{erdosstone} implies any sufficiently large graph of density $d$ will contain $H$, and thus will already be $(H_1, H_2)$-Ramsey before the addition of any random edges.  Hence $p(n; H_1, H_2, d) = 0$ for $d > 1 - 1/k$.
\end{itemize}
In particular, parts (i) and (ii) imply $n^{-2/(t-1)} \le p(n;K_t, K_s,d) \le n^{-2(ts + t - 2s)/(t(t-1)(s+1))}$ for integers $t \ge s \ge 3$ and density $d \le \frac{s-2}{s-1}$.
\end{obs}

The aforementioned result of Krivelevich, Sudakov and Tetali~\cite{kst} shows that the lower bound given by Observation~\ref{obs:trivial}(i) is in fact tight for the $(K_t,K_3)$-Ramsey problem.
\begin{thm}[\cite{kst}] \label{kstthm}
For $0<d \le 1/2$ and $t\geq 3$, we have $p(n;K_t, K_3,d) = n^{-2/(t-1)}$.
\end{thm}

\subsection{Our results}

Krivelevich, Sudakov and Tetali~\cite{kst} raised the question of extending Theorem~\ref{kstthm} to 
 other pairs of graphs $(H_1,H_2)$.  
In this paper, we make the first progress on this question since it was raised in 2006,
 focusing on the cases where the $H_i$ are cliques or cycles.
In our first result we  resolve their question in the case of cliques of size at least $5$.
\begin{thm}\label{thm1}
For $0<d \le 1/2$ and $t\geq s \geq 5$, we have $p(n; K_t, K_s, d) = n^{-1/m_2(K_t, K_{\ceil{s/2}})}$.
\end{thm}
Recall that in the random graph setting one needs $\Theta(n^{2-1/m_2(K_t, K_{s})})$ random edges to ensure $G(n,p)$ is with high probability $(K_t, K_s)$-Ramsey. On the other hand, Theorem 1.6 shows that adding just $\omega(n^{2 - 1/m_2(K_t, K_{\ceil{s/2}})})$ random edges is already enough to make any dense $n$-vertex graph $(K_t, K_s)$-Ramsey. (In fact, in the proof  we see that we only require $\Theta(n^{2-1/m_2(K_t, K_{\ceil{s/2}})})$ random edges.)

To see that this is best possible, consider the complete balanced bipartite graph $G$ on $n$ vertices. Indeed, if $G(n,p)$ is such that there is a $2$-colouring without a red $K_t$ or blue $K_{\ceil{s/2}}$, then by further colouring all the edges in $G$ blue, one obtains a $2$-colouring of $G \cup G(n,p)$ without a red $K_t$ or blue $K_s$.

As one might expect, if one starts with an even denser graph $G$, one needs less randomness to be $(K_t, K_s)$-Ramsey, and we prove a stronger version of Theorem~\ref{thm1} that also gives exact (or log-asymptotically exact) results for larger densities $d$.
\begin{thm} \label{thm:bigcliques}
Given an integer $k \ge 2$, let $d$ be such that $1 - 1/(k-1) < d \le 1 - 1/k$, and let $s$ and $t$ be fixed integers with $2k + 1 \le s \le t$.  If
\begin{itemize}
	\item[(i)] $k = 2$ (that is, $0 < d \le 1/2$), or
	\item[(ii)] $s \equiv 1 \pmod{k}$,
\end{itemize}
then $p(n; K_t, K_s, d) = n^{-1/m_2(K_t, K_{\ceil{s/k}})}$.  Otherwise, if
\begin{itemize}
	\item[(iii)] $k \ge 3$ and $s \not \equiv 1 \pmod{k}$,
\end{itemize}
we have $p(n; K_t, K_s, d) = n^{-(1-o(1))/m_2(K_t, K_{\ceil{s/k}})}$.
\end{thm}

Comparing this result to the  bounds in Observation~\ref{obs:trivial}, we note that in these cases, the thresholds lie strictly between the lower and upper bound, unlike when $s = 3$.  Moreover, perhaps surprisingly, we see that if $s < s' \le t$ and $\ceil{s/k} = \ceil{s'/k}$, the perturbed Ramsey threshold probabilities for the pairs $(K_t,K_s)$ and $(K_t,K_{s'})$ are essentially the same.

In light of Theorems~\ref{kstthm} and~\ref{thm1}, in the setting of cliques, the Krivelevich--Sudakov--Tetali question now remains open only in the 
$(K_t,K_4)$ case. As we will see in the next proposition though, this problem exhibits a different behaviour to the other cases.

\begin{prop}\label{prop:K4vsclique}
For $k \ge 2$, let $1 - 1/(k-1) < d \le 1 - 1/k$.  If $k+2 \le s \le 2k$ and $t \ge s$, we have $n^{-2t/(t(t-1) + \ceil{t/a})} \le p(n;K_t,K_s,d) \le n^{-2/t}$, where $a$ is the smallest integer for which $R(K_{a+1},K_{s-k}) > k$.
\end{prop}

It is worth noting that the formula in Theorem~\ref{thm:bigcliques}, had it also been valid for $k+2 \leq s \le 2k$, would have implied the lower bound of $n^{-2/(t-1)}$ from Observation~\ref{obs:trivial} is (essentially) correct.  While Theorem~\ref{kstthm} shows this to be the case for $s = 3$, the lower bound in Proposition~\ref{prop:K4vsclique} is higher, highlighting that the threshold probability truly does behave differently when $k+2 \le s \le 2k$.  The upper bound in Proposition~\ref{prop:K4vsclique} also represents an improvement over the upper  bound from Observation~\ref{obs:trivial} .

\medskip

\paragraph{\textbf{Remark}}  After making our manuscript available online, we learnt of the simultaneous and  independent work of Powierski~\cite{power}.  He proves Theorem~\ref{thm1} in the case  $s = t \ge 5$ odd, and improves the lower bound of Proposition~\ref{prop:K4vsclique} when $k = 2$ and $s = t = 4$.  In particular, combined with our upper bound from Proposition~\ref{prop:K4vsclique} this shows $p(n; K_4, K_4, d) = n^{-1/2}$ for $0 < d \le 1/2$.
Thus, the question from~\cite{kst} is now resolved for the $(K_t,K_t)$-Ramsey problem for all $t \geq 3$. We suspect though that resolving the $(K_t,K_4)$-Ramsey problem for all $t \geq 5$ will be a significant challenge.

\medskip

We next turn our attention to cycles, completely determining the perturbed Ramsey thresholds for all pairs of cycles and all densities.

\begin{thm} \label{thm:twocycles}
Let $k, \ell \ge 3$ be integers such that either $k$ is odd and $\ell$ is even, or they have the same parity and $k \le \ell$, and let $d \in (0,1)$.  There exist $d_1 = d_1(k,\ell), d_2 = d_2(k,\ell) \in [0,1]$ such that
\[ p(n; C_k, C_{\ell}, d) = \begin{cases}
	n^{-1} & \textrm{if } 0 < d \le d_1, \\
	n^{-2} & \textrm{if } d_1 < d \le d_2, \\
	0 & \textrm{if } d_2 < d.
\end{cases} \]
Moreover, the values of $d_1$ and $d_2$ are as given below.
\begin{center}
\begin{tabular}{c|c|ccc}
& $k$ even & \multicolumn{3}{c}{$k$ odd} \\
& $\ell$ even & $\ell$ even & $\ell$ = 3 & $\ell \ge 5$ odd \\ \hline
$d_1(k,\ell)$ & $0$ & $1/2$ & $1/2$ & $1/2$ \\
$d_2(k,\ell)$ & $0$ & $1/2$ & $4/5$ & $3/4$ \\ 
\end{tabular}
\end{center}
\end{thm}

Theorem~\ref{thm:twocycles} shows there are at most three phases: an initial phase, where a linear number of random edges is needed to make a dense graph $(C_k, C_{\ell})$-Ramsey, an intermediate phase, where it suffices to add a large constant number of edges, and a supercritical phase, where the underlying graph is dense enough to already be $(C_k,C_{\ell})$-Ramsey.  The parities and, in some cases, lengths of the two cycles in question determine at which densities (if at all) the transitions between these phases occur.
In all cases though, our result demonstrates that one needs significantly fewer random edges for the perturbed $(C_k,C_{\ell})$-Ramsey question compared to the analogous result in the random graph setting~\cite{kreu}.

We have thus far focused on the perturbed Ramsey thresholds for pairs of graphs, and our next observation explains why this is the case of greatest interest.  Indeed, for the graph pairs we have studied, our results show that, when adding random edges to a graph of positive density, one needs significantly less randomness to make the graph $(H_1, H_2)$-Ramsey than in $G(n,p)$.  However, this is not the case when there are three or more colours; since the probability threshold for $G(n,p)$ to be $(H,r)$-Ramsey does not depend on $r$, one needs a much denser base graph before requiring less randomness in the perturbed model.  For simplicity we consider only the symmetric case, but similar remarks can be made in a more general setting.

\begin{obs} \label{obs:multicol}
For a graph $H$ and $r \ge 3$, set $k := \min \{ \chi(F) : F \textrm{ is } (H,r-2)\textrm{-Ramsey} \}$.  Then for $d \le 1 - 1/(k-1)$, we have $p(n;r,H,d) = n^{-1/m_2(H)}$.

Indeed, if $p \ge C n^{-1/m_2(H)}$ for some constant $C$, Theorem~\ref{randomramsey} shows that $G(n,p)$ itself will be $(H,r)$-Ramsey, and hence this is an upper bound on the perturbed Ramsey threshold.

For the lower bound, take $G$ to be a complete balanced $(k-1)$-partite graph.  By definition of $k$, we can $(r-2)$-colour the edges of $G$ without creating a monochromatic copy of $H$.  We use the remaining two colours on the edges of $G(n,p)$.  By Theorem~\ref{randomramsey}, if $p \le c n^{-1/m_2(H)}$ for some constant $c$, then with high probability $G(n,p)$ is not $(H,2)$-Ramsey.  This thus gives an $r$-colouring of $G \cup G(n,p)$ without a monochromatic copy of $H$.
\end{obs}

Despite this, in our next result we consider the symmetric multicoloured perturbed Ramsey thresholds for long cycles.  As predicted by Observation~\ref{obs:multicol}, there is a large subcritical regime, where the perturbed thresholds are the same as those from $G(n,p)$.  However, once the underlying graph is very dense, we observe some similarities to Theorem~\ref{thm:twocycles}.  There remains one range of densities where we have been unable to determine the perturbed threshold.

\begin{thm} \label{thm:manycycles}
For a number of colours $r \ge 3$ and odd cycle length $\ell \ge 2^r + 1$, we have
\[ p(n; r,C_{\ell}, d) \begin{cases}
	= n^{-1 + 1/(\ell - 1)} & 0 \le d \le 1 - 2^{-r + 2}, \\
	\in [n^{-1}, n^{-1 + 1/(\ell - 1)}] & 1 - 2^{-r+2} < d \le 1 - 2^{-r+1}, \\
	= n^{-2} & 1 - 2^{-r+1} < d \le 1 - 2^{-r}, \\
	= 0 & 1 - 2^{-r} < d \le 1.
	\end{cases} \]
\end{thm}

Finally, we combine these different settings, determining the perturbed Ramsey thresholds for odd cycles versus cliques.  Our last result provides a different extension of Theorem~\ref{kstthm}, showing that the same perturbed Ramsey threshold remains valid if $K_3$ is replaced with any odd cycle.

\begin{thm} \label{thm:cliquevscycle}
For any clique size $t \ge 4$, odd cycle length $\ell \ge 5$, and density $0 < d \le 1/2$, we have $p(n; K_t, C_{\ell}, d) = n^{-2/(t-1)}$.
\end{thm}

We note that the missing cases are covered by previous results; when $\ell = 3$, the threshold is given by Theorem~\ref{kstthm}, while the case $t = 3$ is covered by Theorem~\ref{thm:twocycles}.
Theorem~\ref{thm:cliquevscycle} is perhaps of additional interest as it is the only example of the perturbed Ramsey question where we know the threshold precisely, but the analogous question in the setting of random graphs remains open.

\subsection{Notation and organisation of the paper}
Let $G$ be a graph. We define $V(G)$ to be the vertex set of $G$ and $E(G)$ to be the edge set of $G$. 
Define $v(G):=|V(G)|$ and $e(G):=|E(G)|$.  For each $x \in V(G)$, we define the \textit{neighbourhood of $x$ in $G$} to be $N_G(x):= \{y\in V(G): xy \in E(G)\}$ and define $d_G(x) := |N_G(x)|$. We define $\Delta (G)$ to be the \emph{maximum degree of $G$}; that is, the maximum value of $d_G(x)$ over all $x \in V(G)$.  Given $X \subseteq V(G)$ we write $N_G(X):= \cap _{x \in X} N_G(x)$ for the \emph{common neighbourhood of $X$ in $G$}.

Let $X \subseteq V(G)$. Then $G[X]$ is the \textit{graph induced by $X$ on $G$} and has vertex set $X$ and edge set $E_G(X) := \{xy \in E(G): x,y \in X\}$. Similarly, if $A, B \subseteq V(G)$ are disjoint, we write $G[A,B]$ for the bipartite graph with vertex classes $A$ and $B$ and edge set $E_G(A,B) := \{xy \in E(G): x \in A, y \in B \}$, and define $e_G(A,B):= \card{E_G(A,B)}$. We will often drop the subscript $G$ from our notation if the graph under consideration is clear from context.

Given two graphs $G,H$ on the same vertex set $V$ we write $G\cup H$ for the graph with vertex set $V$ and edge set $E(G) \cup E(H)$.  We shall assume the vertex set of an $n$-vertex graph to be $[n] := \{1, 2, \hdots, n\}$, unless otherwise specified.

Given a set $A$ and $k \in \mathbb N$ we denote by $A^k$ the set of all ordered $k$-tupes $(a_1, \hdots, a_k)$ of elements from $A$, while $\binom{A}{k}$ denotes the set of all (unordered) $k$-element subsets $\{a_1, \hdots, a_k\}$ of $A$.

Suppose $H_1,\dots, H_s$ and $H'_1,\dots, H'_t$ are graphs. We write $R(\{H_1,\dots, H_s \}, \{H'_1,\dots, H'_t \})$ for the smallest $n \in \mathbb N$ such that whenever $K_n$ is $2$-coloured, there is a red copy of \emph{some} $H_i$ (with $1\leq i \leq s$) or there is a blue copy of \emph{some} $H'_j$ (with $1\leq j \leq t$).

We write $0 < a \ll b \ll c < 1$ to mean that we can choose the constants $a,b,c$ from right to left, with each constant sufficiently small with respect to those preceding it. More precisely, there exist non-decreasing functions $f: (0,1] \to (0,1]$ and $g: (0,1] \to (0,1]$ such that for all $b \leq f(c)$ and $a \leq g(b)$ our calculations and arguments in our proofs are correct. Thus $a \ll b$ implies that we may assume e.g. $a < b$ or $a < b^2$, as needed.  Hierarchies of different lengths are defined similarly.

We will also use the standard asymptotic notation.  Given two positive functions $\alpha, \beta: \mathbb{N} \rightarrow \mathbb{R}$, we write $\alpha = o(\beta)$ and $\beta = \omega(\alpha)$ if $\lim_{n \rightarrow \infty} \alpha / \beta= 0$.  We emphasise that while in certain texts, $\alpha = o(\beta)$ and $\alpha \ll \beta$ are used interchangeably, that is not the case here.

\smallskip

The paper is organised as follows. In the next section we introduce useful tools concerning structures in dense, random, and randomly perturbed graphs. Then in Section~\ref{sec:proofs} we give the proofs of all of our main results. Some concluding remarks are given in Section~\ref{conclude}.  Finally, we include proofs of the less standard results from Section~\ref{sec:tools} in Appendix~\ref{sec:app}.

\section{Useful tools} \label{sec:tools}

\begin{center}
``If I have seen further it is by standing on the shoulders of giants."
\end{center}
\begin{flushright}
--- Isaac Newton (1676)
\end{flushright}

In this section we collect several of the tools we shall use in proving our results.  We shall need some results for finding useful structures in the dense underlying graph $G$ and others to analyse the random edges from $G(n,p)$.

\subsection{Structures in dense graphs}\label{21}

When working with the dense graph, our main tool will be the famous Szemer\'edi Regularity Lemma~\cite{szemeredi}.  We recall some of the key definitions and facts here, omitting many details that can be found in the survey of Koml\'os and Simonovits~\cite{ks}.  We start with that of an $\eps$-regular pair, which is a pair of vertex sets that is very regular in terms of edge densities between linear-sized subsets, as is characteristic of bipartite random graphs.

\begin{defin} \label{def:epsregular}
Given $\eps > 0$, a graph $G$ and two disjoint vertex sets $A, B \subset V(G)$, the pair $(A,B)$ is $\eps$-regular if for every $X \subseteq A$ and $Y \subseteq B$ with $\card{X} > \eps \card{A}$ and $\card{Y} > \eps \card{B}$, we have $\card{d(X,Y) - d(A,B)} < \eps$, where $d(S,T) := e(S,T)/(\card{S}\card{T})$ for any vertex sets $S$ and $T$.
\end{defin}

This strong regularity condition implies many desirable properties, some of which are collected in the following lemma.  The first states that almost all $\ell$-sets of vertices from the same side of an $\eps$-regular pair will have a large common neighbourhood, while the second states that $\eps$-regularity is essentially inherited by subsets.  We omit the proofs of these standard facts.

\begin{lemma} \label{lem:regpairs}
Let $(A,B)$ be an $\eps$-regular pair in a graph $G$ with $d(A,B) = d$.
\begin{itemize}
	\item[(i)] If $\ell \ge 1$ and $(d - \eps)^{\ell - 1} > \eps$, then
	\[ \card{ \left\{ (x_1, x_2, \hdots, x_{\ell}) \in A^{\ell} : \card{ \cap_i N(x_i) \cap B } \le (d - \eps)^{\ell} \card{B} \right\} } \le \ell \eps \card{A}^{\ell}. \]
	\item[(ii)] If $\alpha > \eps$, and $A' \subset A$ and $B' \subset B$ satisfy $\card{A'} \ge \alpha \card{A}$ and $\card{B'} \ge \alpha \card{B}$, then $(A',B')$ is an $\eps'$-regular pair of density $d'$, where $\eps':= \max \{ \eps / \alpha, 2 \eps \}$ and $\card{d' - d} < \eps$.
\end{itemize}
\end{lemma}

We can now state the Regularity Lemma itself.  We present the multicoloured version, as given in~\cite{ks}.

\begin{thm} \label{thm:reglemma}
For any $\eps > 0$ and $r,t \in \mathbb{N}$, there are $T = T(\eps, r, t)$ and $n_0 = n_0(\eps, r, t)$ such that if $n \ge n_0$ and the edges of an $n$-vertex graph $G$ are $r$-coloured, with $G_{\ell}$ representing the $\ell$th colour class, the vertex set $V(G)$ can be partitioned into sets $V_0, V_1, \hdots, V_k$ for some $t \le k \le T$, such that $\card{V_0} < \eps n$, $\card{V_1} = \card{V_2} = \hdots = \card{V_k}$, and all but at most $\eps k^2$ pairs $(V_i, V_j)$, $1 \le i < j \le k$, are $\eps$-regular pairs in each of the subgraphs $G_{\ell}$ simultaneously.
\end{thm}

Informally speaking, the lemma says that, apart from a small exceptional set, the vertices of $G$ can be partitioned into a large but bounded number of parts, such that between almost all pairs of parts the edges of each colour seem randomly distributed.  We shall often apply this in the form of the following corollary, which follows by combining the Regularity Lemma with Tur\'an's Theorem~\cite{turan} (see~\cite{ks} for details).

\begin{cor} \label{cor:regtuple}
For every $r \ge 1$ and $\eps , \delta> 0$ with $\delta \geq 3 \eps$ there is some $\eta = \eta(\eps,\delta, r) > 0$ and $n_0=n_0(\eps,\delta, r)\in \mathbb N$ such that the following holds for all $n \geq n_0$ and $k \ge 2$.
If  $G$ is an $r$-coloured $n$-vertex graph of density at least $1 - 1/(k-1) + \delta$, then there are pairwise disjoint vertex sets $V_1, V_2, \hdots ,V_k \subset V(G)$ such that $\card{V_1} = \hdots = \card{V_k} \ge \eta n$, and, for each $1 \le i < j \le  k$, there is some colour $c_{i,j} \in [r]$ for which the edges between $V_i$ and $V_j$ of colour $c_{i,j}$ form an $\eps$-regular pair of density at least $\delta/(2r)$.
\end{cor}

We will use the structure from Corollary~\ref{cor:regtuple} to  build monochromatic cliques.  We will be able to do so because, as given by Lemma~\ref{lem:regpairs}, almost all $\ell$-tuples of vertices in $V_i$ will have large common neighbourhoods in each of the other parts $V_j$.  However, in some of our applications we will require \emph{all} $\ell$-tuples to have large common neighbourhoods.  For that we make use of another powerful tool from extremal combinatorics, dependent random choice.  The following lemma, proven in Appendix~\ref{sec:multipartDRCproof}, is a multipartite version of the basic lemma from the survey of Fox and Sudakov~\cite{fs}.

\begin{lemma} \label{lem:multipartDRC}
Given $s \in \mathbb{N}$ and $\delta > 0$, suppose $0 < \eps \le \min\{ 1/(2s), \delta/2 \}$.  Let $V_1, V_2, \hdots, V_s$ be disjoint sets of vertices from a graph $G$, each of size $m$, such that for all $i \in [s-1]$, the pair $(V_i, V_s)$ is $\eps$-regular of density at least $\delta$.  For every $\beta > 0$ and $\ell \in \mathbb{N}$ there is some $\gamma = \gamma(s, \delta, \beta, \ell) > 0$ such that, if $m$ is sufficiently large, there is a subset $U_s \subseteq V_s$ of size at least $m^{1 - \beta}$ such that every set of $\ell$ vertices from $U_s$ has at least $\gamma m$ common neighbours in each of the sets $V_i$, $i \in [s-1]$.
\end{lemma}

When we are instead seeking slightly sparser structures within our $\eps$-regular $k$-tuples, we shall use the following well-known result, a special case of the so-called Key Lemma (see~\cite{ks}).

\begin{lemma}  \label{keylem}
Let $H$ be a fixed graph, $r\geq 2$, $0 < \varepsilon \ll  d$, and let $m\in \mathbb N$ be sufficiently large. Suppose $R$ is a graph on $[r]$ and suppose that $G$ is a graph with vertex classes $V_1,\dots, V_r$ so that
\begin{itemize}
	\item[(i)] $|V_i|=m$ for all $i \in [r]$, and
	\item[(ii)] $(V_i,V_j)$ forms an $\eps$-regular pair of density at least $d$ in $G$ for all $1\leq i<j \leq r$ with $\{i,j\} \in E(R)$.
\end{itemize}
If there is a homomorphism from $H$ to $R$ 
then $G$ contains a copy of $H$. 
\end{lemma}

\subsection{Properties of random graphs}\label{22}

The previous lemma guaranteed the existence of fixed subgraphs within deterministic graphs.  When we are instead dealing with the random graph $G(n,p)$, we shall apply the following result, a standard application of the Janson inequality~\cite[Theorem 2.14]{jlr}, whose proof is given in Appendix~\ref{sec:jansonproof}.

\begin{thm} \label{thm:janson}
Let $H$ be a graph with $v \ge 2$ vertices and $e \ge 1$ edges.  Let $[n]$ be the vertex set of $G(n,p)$, and, for some $\xi > 0$, let $\mc H$ be a collection of $\xi n^v$ possible copies of $H$ supported on $[n]$.  The probability that $G(n,p)$ does not contain any copy of $H$ from $\mc H$ is at most $\exp(- \xi \mu_1 / ( 2^{v+1} v!))$, where $\mu_1=\mu_1 (H) := \min \{ n^{v(F)} p^{e(F)} : F \subseteq H, e(F) \ge 1 \}$.
\end{thm}

However, for our purposes, the most important properties of random graphs will be their Ramsey properties.  As stated in the introduction, Conjecture~\ref{conjkreu} of Kohayakawa and Kreuter~\cite{kreu} suggests what the threshold for the random graph $G(n,p)$ being $(H_1, \hdots, H_r)$-Ramsey should be.  While the 1-statement is now known to hold in general~\cite{mns}, the corresponding 0-statement is only known for cycles~\cite{kreu} and cliques~\cite{msss}.

\begin{thm}[\cite{kreu, msss}] \label{thm:randramlower}
Let $r \ge 2$, and let $H_1, H_2, \hdots, H_r$ be graphs such that either every $H_i$ is a complete graph or every $H_i$ is a cycle, and $m_2(H_1) \ge m_2(H_2) \ge \hdots \ge m_2(H_r)$.  There is some constant $c > 0$ such that if $p \le c n^{-1/m_2(H_1, H_2)}$, then with high probability $G(n,p)$ is not $(H_1, H_2, \hdots, H_r)$-Ramsey.
\end{thm}

In our applications, we require slightly stronger variants of the $1$-statement, and we define the properties we shall need below.

\begin{defin}[Robust and global Ramsey properties]
Let $H_1$ and $H_2$ be two fixed graphs, and let $G$ be an $n$-vertex graph on the vertex set $[n]$.

Given, for $i \in [2]$, families $\mc F_i \subseteq \binom{[n]}{v(H_i)}$ of \emph{forbidden} subsets of $v(H_i)$ vertices, we say $G$ is \emph{robustly $(H_1, H_2)$-Ramsey with respect to $(\mc F_1, \mc F_2)$} if every $2$-colouring of $G$ contains a red copy of $H_1$ or a blue copy of $H_2$, such that the vertex set of the monochromatic subgraph is not forbidden.

Given $\mu >0$, we say that $G$ is \emph{$\mu$-globally $(H_1, H_2)$-Ramsey} if, for every $2$-colouring of $G$ and for every subset $U \subseteq [n]$ of at least $\mu n$ vertices, $G[U]$ contains a red copy of $H_1$ or a blue copy of $H_2$.
\end{defin}

It turns out for a wide collection of pairs of graphs $H_1,H_2$, above the threshold from Conjecture~\ref{conjkreu}, $G(n,p)$ is not only $(H_1, H_2)$-Ramsey, but robustly and globally so.  While~\cite{kreu} does not give explicit bounds on the error terms, making it a little harder to verify that this strengthening is possible, the more recent containers-based proofs given by Gugelmann, Nenadov, Person, \v{S}kori\'c, Steger and Thomas~\cite{gnpsst} and Hancock, Staden and Treglown~\cite{hst} allow for the necessary extensions.  In Appendix~\ref{sec:global1ramseyproof} we prove Theorem~\ref{thm:global1ramsey}, implementing the modifications that must be made to the existing proofs.

\begin{thm}[\cite{gnpsst,hst}] \label{thm:global1ramsey}
Let $H_1$ and $H_2$ be graphs such that $m_2(H_1) \ge m_2(H_2) \ge 1$, and 
\begin{itemize}
	\item[(a)] $m_2(H_1) = m_2(H_2)$, or 
	\item[(b)] $H_1$ is strictly balanced with respect to $m_2(\cdot, H_2)$.
\end{itemize}

The random graph $G(n,p)$ then has the following Ramsey properties:
\begin{itemize}
	\item[(i)] There are constants $\gamma = \gamma(H_1, H_2) > 0$ and $C_1 = C_1(H_1, H_2)$ such that if $p \ge C_1 n^{-1/m_2(H_1, H_2)}$ and, for $i \in [2]$, $\mc F_i \subseteq \binom{[n]}{v(H_i)}$ is a collection of at most $\gamma n^{v(H_i)}$ forbidden subsets, then $G(n,p)$ is with high probability robustly $(H_1, H_2)$-Ramsey with respect to $(\mc F_1, \mc F_2)$.
	\item[(ii)] For every $\mu > 0$ there is a constant $C_2 = C_2(H_1, H_2, \mu)$ such that if $p \ge C_2 n^{-1/m_2(H_1, H_2)}$, then $G(n,p)$ is with high probability $\mu$-globally $(H_1, H_2)$-Ramsey.
	\item[(iii)] If we further have $m_2(H_2) > 1$, then there are constants $\beta_0 = \beta_0(H_1, H_2) > 0$ and $C_3 = C_3(H_1, H_2)$ such that if $0 \le \beta \le \beta_0$ and $p \ge C_3 n^{-(1-\beta)/m_2(H_1, H_2)}$, then with high probability $G(n,p)$ is $n^{-\beta}$-globally $(H_1, H_2)$-Ramsey.
\end{itemize}
\end{thm}

Before we proceed, it is worth comparing the exponents in these different thresholds for pairs of cliques.  For $t \ge 4$, we have $m_2(K_{t-1},K_{t-1}) = \frac{t}{2} < \frac{t(t-1)}{2t-3} = m_2(K_t, K_3) < m_2(K_t,K_4) < \hdots < m_2(K_t, K_t) = \frac{t+1}{2}$.  In particular, when $G(n,p)$ is $(K_t, K_3)$-Ramsey, it will also be $(K_k, K_{\ell})$-Ramsey for $k,\ell < t$.

\subsection{Properties of randomly perturbed graphs} \label{23}

While some of the previous results allowed us to find subgraphs within dense or random graphs, we will sometimes need the existence of certain graphs in the randomly perturbed model.  Krivelevich, Sudakov and Tetali~\cite{kst} showed that the threshold probabilities for this problem depend on the sparsest partitions of the desired graph, and the final result we shall make use of is their $1$-statement.

\begin{thm}[\cite{kst}] \label{thm:perturbedturan}
Given a graph $F$, define $\rho(F) := \max \{ e(F')/v(F') : F' \subseteq F, v(F') \ge 1 \}$, and, for $k \ge 2$, set
\[ \rho_k(F) := \min_{V(F) = \cup_i V_i} \max_i \rho(F[V_i]), \]
where the minimum is taken over all partitions into at most $k$ parts.  If $d > 1 - 1/(k-1)$, $p = \omega\left( n^{-1/\rho_k(F)} \right)$, and $G$ is an $n$-vertex graph of density $d$, then $F \subseteq G \cup G(n,p)$ with high probability.
\end{thm}

\section{Proofs} \label{sec:proofs}

\begin{center}
``Your faith was strong but you needed proof."
\end{center}
\begin{flushright}
--- Leonard Cohen, \emph{Hallelujah} (1984)
\end{flushright}

With these tools at our disposal, we are now ready to prove our main results, establishing perturbed Ramsey thresholds for various pairs of graphs.

\subsection{Proof of Theorem~\ref{thm:bigcliques}}

Our first result establishes perturbed Ramsey thresholds for pairs of  cliques that are not too small.  Specifically, let $k \ge 2$ be such that $1 - 1/(k-1) < d \le 1 - 1/k$, and suppose $2k+1 \le s \le t$.  For convenience, we set $\ell := \ceil{s/k}$, and note that $\ell \ge 3$.  We shall show $p(n; K_t,K_s,d) = n^{-1/m_2(K_t, K_{\ell})}$ if $k = 2$ or $s \equiv 1 \pmod{k}$, and otherwise obtain this same threshold log-asymptotically.  In all cases, the lower bound follows from the same argument, which we now present.

\begin{proof}[Proof of Theorem~\ref{thm:bigcliques} (lower bound)]
Let $G$ be the $k$-partite $n$-vertex Tur\'an graph with vertex classes $V_1, V_2, \hdots, V_k$, which has density at least $1 - 1/k$, and let $p \le c n^{-1/m_2(K_t, K_{\ell})}$, where the constant $c$ is as in Theorem~\ref{thm:randramlower}.  We need to show that, with high probability, the edges of $G \cup G(n,p)$ can be $2$-coloured without creating a red $K_t$ or a blue $K_s$.  We colour the edges of $G$ blue, leaving only the edges of $G(n,p)[V_i]$, $i \in [k]$, uncoloured.  By Theorem~\ref{thm:randramlower}, with high probability $G(n,p)$ can be $2$-coloured without creating a red $K_t$ or a blue $K_{\ell}$.  This then gives the desired colouring --- the connected components of the red subgraph lie within the red subgraphs of $G(n,p)[V_i]$, which are $K_t$-free, while, since the largest blue clique in each $G(n,p)[V_i]$ has at most $\ell - 1$ vertices, the largest blue clique in $G \cup G(n,p)$ has size at most $k(\ell - 1) < s$.  Hence $p(n;K_t, K_s, d) \ge n^{-1/m_2(K_t, K_{\ell})}$.
\end{proof}

We divide the proof of the upper bound into three parts, treating each case in turn.

\subsubsection{Proof of Theorem~\ref{thm:bigcliques} (upper bound), Case (i)}

We start with the case $k = 2$.  Let $G$ be a graph of density $4d > 0$.  By applying Corollary~\ref{cor:regtuple} with $r = 1$, we can find an $\eps$-regular pair $(U,W)$ of density at least $2d$, where $\eps = \eps(d)>0$ is sufficiently small and $\card{U} = \card{W} = m := \eta n$ for some $\eta = \eta(d) > 0$.

Set $\xi := d^{\ell} / (2 \ell)$, let $\mu := \xi d / 2$, and recall that $m_2(K_t, K_{\ell}) > m_2(K_{t-1}, K_{t-1})$.  Thus, if $C \ge \max \{ C_{2}(K_t, K_{\ell}, \mu \eta), C_{2}(K_{t-1},K_{t-1},\mu \eta) \}$, it follows from Theorem~\ref{thm:global1ramsey} that if $p \ge C n^{-1/m_2(K_t, K_{\ell})}$, then with high probability,
\begin{itemize}
	\item[(a)] $G(n,p)$ is both $\mu \eta$-globally $(K_t, K_{\ell})$-Ramsey and $\mu \eta$-globally $(K_{t-1}, K_{t-1})$-Ramsey.
\end{itemize}

Furthermore, set $\mc F_1 = \emptyset$, and let $\mc F_2 \subseteq \binom{U}{\ell}$ be those $\ell$-sets with fewer than $d^{\ell} m$ common neighbours in $W$.  By Lemma~\ref{lem:regpairs}(i), it follows that $\card{\mc F_2} \le \ell \eps m^{\ell}$.  Since $\eps$ was chosen to be sufficiently small, we can ensure $\ell \eps < \gamma(K_t, K_{\ell})$, where $\gamma$ is as in Theorem~\ref{thm:global1ramsey}.  Since $G(n,p)[U] \sim G(m,p)$, provided $C \ge C_{1}(K_t, K_{\ell}) \eta^{-1/m_2(K_t, K_{\ell})}$ as well, with high probability we also have, by Theorem~\ref{thm:global1ramsey},
\begin{itemize}
	\item[(b)] $G(n,p)[U]$ is robustly $(K_t, K_{\ell})$-Ramsey with respect to $(\mc F_1, \mc F_2)$.
\end{itemize}

We may therefore assume $G(n,p)$ has Properties (a) and (b), and shall show that this implies $G \cup G(n,p)$ is $(K_t, K_s)$-Ramsey.  Suppose for a contradiction that $G \cup G(n,p)$ has a $2$-colouring with neither a red $K_t$ nor a blue $K_s$.  We first need the following claim.

\begin{claim} \label{clm:fewrededges}
We may assume that no vertex in $U$ has more than $\xi m$ red edges to $W$.
\end{claim}

\begin{proof}
First suppose $s \le t-1$.  If we have a vertex $u \in U$ with at least $\xi m$ red neighbours in $W$, let $Y \subseteq W$ be the set of those red neighbours.  By Property (a), we find a red $K_{t-1}$ in $G(n,p)[Y]$, in which case we can add $u$ to obtain the desired red $K_t$, or we find a blue $K_{t-1}$, which contains the desired blue $s$-clique.

Next suppose $s = t$.  By symmetry, we may assume that at least half the edges of $G$ between $U$ and $W$ are blue.  In particular, since the density of $(U,W)$ is at least $2d$, this means we can find a vertex $v \in W$ and a set $A \subset U$ of $d m$ blue neighbours of $v$.  Now suppose further that there is a vertex $u \in U$ with a set $B \subset W$ of $\xi m$ red neighbours of $u$.

By $\eps$-regularity, there are at least $d \card{A} \card{B}$ edges between $A$ and $B$.  If at least half of these are red, then we can find a vertex $b \in B$ with a set $A' \subseteq A$ of at least $d \card{A} / 2$ red neighbours of $b$.  Since $\card{A'} \ge \mu m$, by Property (a), $G(n,p)[A']$ has a red $K_{t-1}$, which together with $b$ gives a red $K_t$, or a blue $K_{t-1}$, which together with $v$ gives a blue $K_t$.

Hence at least half the edges between $A$ and $B$ are blue, which gives a vertex $a \in A$ with a set $B' \subseteq B$ of at least $d \card{B} / 2$ blue neighbours of $a$.  Again, $\card{B'} \ge \mu m$, and so by Property (a) we find in $G(n,p)[B']$ a red $K_{t-1}$, which extends via $u$ to a red $K_t$, or a blue $K_{t-1}$, to which we can add $a$ to obtain a blue $K_t$.
\end{proof}

With the claim established, we complete our proof.  By Property (b), $G(n,p)[U]$ contains  a red $K_t$, in which case we are done, or a blue $\ell$-clique whose vertex set $S$ does not lie in $\mc F_2$.  In particular, this implies that $S$ has at least $d^{\ell} m$ common neighbours in $W$.

We discard those common neighbours with a least one red edge to $S$; by Claim~\ref{clm:fewrededges}, we are left with a set $Y$ of at least $(d^\ell m -\ell \xi m)= d^{\ell} m/ 2$ common blue neighbours of $S$.  By Property (a), $G(n,p)[Y]$ has a red $K_t$, in which case we are done, or a blue $K_{\ell}$, which together with $S$ gives the desired blue $s$-clique.

Thus, when $k = 2$, we have $p(n;K_t, K_s, d) \le n^{-1/m_2(K_t, K_{\ell})}$, matching our lower bound.\qed 

\subsubsection{Proof of Theorem~\ref{thm:bigcliques} (upper bound), Case (ii)}

We next consider the case $k \ge 3$, where we provide the exact result if $s \equiv 1 \pmod{k}$.  Specifically, we wish to show that if $G$ is a graph of density $d = 1 - 1/(k-1) + \delta$, where $\delta > 0$, then if $p \ge C n^{-1/m_2(K_t, K_{\ell})}$ for some appropriately large constant $C$, with high probability $G \cup G(n,p)$ will be $(K_t, K_s)$-Ramsey.

Our strategy bears some resemblance to the previous proof.  In Case (i), we built a blue $K_s$ in two stages, finding half in each part of an $\eps$-regular pair.  When $k \ge 3$ and $s = k(\ell - 1) + 1$, we instead take $k$ steps to build the blue clique within an $\eps$-regular $k$-tuple, finding a blue $K_{\ell - 1}$ in each of the first $k-1$ parts and a blue $K_{\ell}$ in the last part.

When we find a set of $\ell - 1$ vertices from a part to add to the growing blue clique, we will need to ensure that they have a sufficiently large neighbourhood in each of the remaining parts, so that the process may continue.  We will do this through repeated use of Lemma~\ref{lem:multipartDRC}.\footnote{In Case (i), we identified a forbidden set of vertex sets with small neighbourhoods, and then used the robust Ramsey properties of $G(n,p)$ to ensure the monochromatic $K_{\ell}$ had a large neighbourhood.  However, in this setting the $\eps$-regular pairs we have depend on the vertices chosen earlier, and our error bounds do not allow for so many applications of the robust Ramsey property.}  However, this leaves us with a sublinear set in which we will have to find a blue clique, and so we will need strong global Ramsey properties.

To that end, recall that if a pair $(q,r)$ with $q \ge r$ precedes $(t,\ell)$ lexicographically (that is, either $q < t$, or $q = t$ and $r < \ell$), then $m_2(K_q,K_r) < m_2(K_t,K_{\ell})$.  We can therefore fix some $\beta > 0$ such that, for all such pairs $(q,r)$, we have $m_2(K_q, K_r) \le (1 - \beta) m_2(K_t, K_{\ell})$ and $\beta \le \beta_0(K_q, K_r)$, where $\beta_0$ is as in Theorem~\ref{thm:global1ramsey}.

We now define a sequence of constants that we shall use in the sequel.  To start, let $\delta_0 := \delta$, and set $\delta_1 := \delta_0/4$.  For $i \in [k-2]$, set $\delta_{i+1} := \delta_i / 2$.  For some soon-to-be-determined $\eps_1$, let $\eta := \eta(\delta_1, \eps_1,2)$, where $\eta$ is as in Corollary~\ref{cor:regtuple}, and set $m_1 := \eta n$.  For $i \in [k-1]$, let $\gamma_i := \gamma(k-i + 1, \delta_i, \beta/2, t-1)$, where $\gamma$ is as in Lemma~\ref{lem:multipartDRC}, and set $m_{i+1} := \ceil{\gamma_i m_i}$.  Note that the $m_i$ are linear in $n$, and set $\mu := m_k / n$.

Further, for $i \in [k-2]$, set $\eps_{i+1} := \max \{ 2 \eps_i, \eps_i / \gamma_i \}$.  Noting that the $\delta_i$ and $\gamma_i$ are independent of $\eps_1$, while the $\eps_i$ are linear in $\eps_1$, choose $\eps_1$ sufficiently small so that $\eps_i \le \delta_i/2$ for all $i \in [k-1]$.  We shall further assume that $n$ is large enough for all following calculations and applications of the lemmas to be valid.

With these technicalities out of the way, we can proceed with our proof.  We shall show that $G \cup G(n,p)$ is $(K_t, K_s)$-Ramsey, provided $G(n,p)$ satisfies the following properties.
\begin{itemize}
	\item[(a)] $G(n,p)$ is $n^{-\beta}$-globally $(K_{t-1}, K_{s-1})$-Ramsey,
	\item[(b)] $G(n,p)$ is $n^{-\beta}$-globally $(K_t, K_{\ell - 1})$-Ramsey,
	\item[(c)] $G(n,p)$ is $\mu$-globally $(K_t, K_{\ell})$-Ramsey, and 
	\item[(d)] if $s < t$, then $G(n,p)$ is $n^{-\beta}$-globally $(K_{\ceil{t/k}}, K_s)$-Ramsey.
\end{itemize}

Theorem~\ref{thm:global1ramsey} ensures Properties (a), (c), (d) and, if $\ell \ge 4$, (b) hold with high probability whenever $p \ge C n^{-1/m_2(K_t, K_{\ell})}$ for a large enough constant $C$.  When $\ell = 3$, being $n^{-\beta}$-globally $(K_t, K_2)$-Ramsey is equivalent to every induced subgraph on $n^{1- \beta}$ vertices containing a $t$-clique.  Hence in this case we instead apply Theorem~\ref{thm:janson}, taking a union bound over all such vertex subsets.

\medskip

Now let $G$ and $G(n,p)$ be as above, and suppose for contradiction there is a $2$-colouring of $G \cup G(n,p)$ with neither a red $K_t$ nor a blue $K_s$.  Applying Corollary~\ref{cor:regtuple} with $r = 2$ to the coloured graph $G$, we find pairwise disjoint vertex sets $\vSet11, \vSet21, \hdots, \vSet{k}{1}$, each of size at least $m_1$, such that for each pair $\{i,j\} \in \binom{[k]}{2}$ there is some colour $c_{i,j}$ for which $(\vSet{i}{1}, \vSet{j}{1})$ is $\eps_1$-regular with density at least $\delta_1$ in the colour $c_{i,j}$.

\begin{claim}
All the colours $c_{i,j}$, $\{i,j\} \in \binom{[k]}{2}$, are the same.
\end{claim}

\begin{proof}
Suppose this were not the case.  There are then $i,j_1, j_2$ such that $c_{i,j_1}$ is red while $c_{i,j_2}$ is blue.  Apply Lemma~\ref{lem:multipartDRC} to the triple $\vSet{j_1}{1}, \vSet{j_2}{1}, \vSet{i}{1}$ to find a set $U_i \subset \vSet{i}{1}$ of size $m_1^{1 - \beta/2} > n^{1 - \beta}$ such that every $(t-1)$-set in $U_i$ has at least one common neighbour in both $\vSet{j_1}{1}$ and $\vSet{j_2}{1}$, using red and blue edges respectively (note that our conservative definition of $\gamma_1$ in fact guarantees us linearly many common neighbours in both $V_{j_1}^{(1)}$ and $V_{j_2}^{(1)}$, but here we only require one).

By Property (a), $G(n,p)[U_i]$ must contain a red $K_{t-1}$ or a blue $K_{s-1}$.  Extending this monochromatic clique with a common neighbour in $\vSet{j_1}{1}$ or $\vSet{j_2}{1}$ respectively gives a red $K_t$ or a blue $K_s$ in $G \cup G(n,p)$, a contradiction.
\end{proof}

Let us first assume that $c_{i,j}$ is blue for every pair $\{i,j\}$ (if $s = t$, we may assume this without loss of generality).  We shall either find a red $K_t$ or build a blue $K_s$ in $k$ stages.

Suppose we have already run $a \ge 0$ stages of this process, thereby building a blue $K_{a (\ell - 1)}$ on vertices $W = \cup_{i=1}^a W_{k+1-i}$, where $W_j$ is a set of $\ell - 1$ vertices from $V_j$.  Further, for every $i \in [k-a]$, we have sets $\vSet{i}{a+1}$ of size $m_{a+1}$ contained in the common (blue) neighbourhood of $W$, such that each pair $( \vSet{i}{a+1}, \vSet{j}{a+1} )$, $\{i,j\} \in \binom{[k-a]}{2}$, is $\eps_{a+1}$-regular of density at least $\delta_{a+1}$ in blue.

\medskip

If $a \le k-2$, we apply Lemma~\ref{lem:multipartDRC} to the $(k-a)$-tuple $\vSet{1}{a+1}, \vSet{2}{a+1}, \hdots, \vSet{k-a}{a+1}$.  This gives us a set $U_{k-a} \subset \vSet{k-a}{a+1}$ of size $m_{a+1}^{1 - \beta/2} > n^{1 - \beta}$ (since $m_{a+1}$ is linear in $n$), such that any $(\ell - 1)$-set in $U_{k-a}$ has at least $m_{a+2}$ common (blue) neighbours in each $\vSet{i}{a+1}$, $i \in [k-1-a]$.

By Property (b), $G(n,p)[U_{k-a}]$ contains  a red $K_t$ or a blue $K_{\ell - 1}$.  In the former case we are done, so we may assume there is a blue ($\ell - 1$)-clique on the vertices $W_{k-a} \subset U_{k-a}$.  For each $i \in [k-1-a]$, let $\vSet{i}{a+2} \subset \vSet{i}{a+1}$ be a set of $m_{a+2}$ common blue neighbours of $W_{k-a}$.  By Lemma~\ref{lem:regpairs}(ii), each pair $(\vSet{i}{a+2}, \vSet{j}{a+2})$, $\{i,j\} \in \binom{[k-1-a]}{2}$, is $\eps_{a+2}$-regular with density at least $\delta_{a+1} - \eps_{a+1} > \delta_{a+2}$.

\medskip

In the final stage, when $a = k-1$, we simply set $U_1 := \vSet{1}{k}$, which has size $m_k = \mu n$.  By Property (c), $G(n,p)[U_1]$ contains  a red $K_t$ or a blue $K_{\ell}$.  We are again done in the former case, so we may assume the existence of a blue $\ell$-clique on the vertices $W_1 \subset U_a$.  This then gives a blue $K_s$ on the vertices $\cup_{i=1}^k W_i$, contradicting our assumption that $G \cup G(n,p)$ has no blue $K_s$.

\medskip

On the other hand, if $s < t$ and each colour $c_{i,j}$ is red instead, we follow a very similar process to that above, except in each $U_i$ we use Property (d) to find either a blue $K_s$ or a red $K_{\ceil{t/k}}$ instead.  In this way, we either have a blue $K_s$ or we build a red $K_t$ in the $k$ stages, obtaining the desired contradiction.

Hence we indeed have $p(n;K_t, K_s, d) = n^{-1/m_2(K_t, K_{\ell})}$ in this case as well.\qed

\subsubsection{Proof of Theorem~\ref{thm:bigcliques} (upper bound), Case (iii)}  The third and final case is when $k \ge 3$ and $k(\ell - 1) + 2 \le s \le k \ell$.  In this setting we can only match the lower bound on the perturbed Ramsey threshold log-asymptotically; that is, for any $\beta > 0$, we show that if $G$ is a graph of density $d \ge 1 - 1/(k-1) + \delta$ and $p \ge C n^{- (1 - \beta)/m_2(K_t, K_{\ell})}$, then with high probability $G \cup G(n,p)$ will be $(K_t, K_s)$-Ramsey.

\medskip

The proof is essentially the same as in Case (ii), except when building the blue clique, we will find a blue $K_{\ell}$ in each part $V_i$, rather than just a $K_{\ell - 1}$.  Properties (a) and (d)  hold as before.  We can replace Properties (b) and (c) above with the following.
\begin{itemize}
	\item[(b')] $G(n,p)$ is $n^{-\beta}$-globally $(K_t, K_{\ell})$-Ramsey.
\end{itemize}

Since we may freely assume that $\beta$ is small enough to satisfy $\beta \le \beta_0(K_t, K_{\ell})$, Theorem~\ref{thm:global1ramsey} guarantees $G(n,p)$ satisfies Property (b') with high probability.  We can then run the same process as before, picking up $\ell$ vertices from each $U_i$ to add to the blue clique.  Thus, after $k$ stages, we would have built a blue $K_{k \ell} \supseteq K_s$.

This shows $p(n;K_t,K_s,d) \le n^{-(1- \beta)/m_2(K_t, K_{\ell})}$ for $n$ sufficiently large.  Since $\beta$ can be taken to be arbitrarily small, we have $p(n;K_t, K_s, d) = n^{-(1 - o(1))/m_2(K_t, K_{\ell})}$, thereby completing the proof of Theorem~\ref{thm:bigcliques}.\qed

\subsection{Proof of Proposition~\ref{prop:K4vsclique}}

In Theorem~\ref{thm:bigcliques}, we required $s \ge 2k+1$ or, equivalently, $\ell := \ceil{s/k}\ge 3$.  This condition was necessary to apply Theorem~\ref{thm:global1ramsey}, which asserts that $G(n,p)$ will be globally $(K_t, K_{\ell})$-Ramsey when $p =\omega( n^{-1/m_2(K_t, K_{\ell})})$. Unfortunately, this is not true for $\ell = 2$.  Indeed, being $(K_t,K_2)$-Ramsey is equivalent to containing a copy of $K_t$, and it is well known that the local property of the appearance of $K_t$ in $G(n,p)$ occurs at a lower threshold probability than the global property of every large subset containing a $t$-clique.

This gives hope of improving the lower bound when we are dealing with a smaller clique: rather than our simplistic approach in Theorem~\ref{thm:bigcliques}, where all edges of the dense graph $G$ received the same colour, we might hope to take advantage of the sparseness of the $t$-cliques in $G(n,p)$ to find a cleverer colouring of the edges of $G$, thus making it easier to avoid monochromatic copies of $K_t$ and $K_s$ when $s \le 2k$.  Proposition~\ref{prop:K4vsclique}, despite falling short of determining $p(n;K_t,K_s,d)$, shows that this is indeed the case, and that one can improve upon both the obvious lower and upper bounds when $k + 2 \le s \le 2k$.  We start with the former.

\begin{proof}[Proof of Proposition~\ref{prop:K4vsclique} (lower bound)]
We write $\ell := \ceil{t/a}$ for simplicity.  Let $G$ be the $k$-partite Tur\'an graph with vertex classes $V_1, V_2, \hdots, V_k$.  We shall show that there is some constant $b > 0$ such that if $p \le b n^{-2t / (t(t-1) + \ell)}$, then with high probability $G \cup G(n,p)$ is not $(K_t, K_s)$-Ramsey.

Using a result of Kreuter~\cite{kreu96} concerning asymmetric vertex-Ramsey properties of random graphs, if $p \le b n^{-2t / (t(t-1) + \ell)}$ for some constant $b = b(t,\ell)$, then we can with high probability partition the vertices 
$V (G(n,p))= A \cup B$ such that $G(n,p)[A]$ is $K_t$-free and $G(n,p)[B]$ is $K_{\ell}$-free.  
For $i \in [k]$, let $A_i := A \cap V_i$ and $B_i := B \cap V_i$.

For each $i \in [k]$, colour all edges within $A_i$ and within $B_i$ red.  We further colour all edges from $A_i$ to any other part blue.  Now all that remains are the edges between $B_i$ and $B_j$ for $1 \le i < j \le k$.

Recall that we have $R(a+1,s-k) > k$.  Hence we can find a colouring $\varphi : \binom{[k]}{2} \rightarrow \{ \mathrm{red}, \mathrm{blue} \}$ of $K_k$ with no red clique of size $a+1$ and no blue clique of size $s-k$.  Then, for each $1 \le i < j \le k$, we colour all edges between $B_i$ and $B_j$ with the colour $\varphi(\{i,j\})$.

We claim that this colouring of $G \cup G(n,p)$ has neither a red $K_t$ nor a blue $K_s$.  First consider the red subgraph.  Each part $A_i$ is disconnected from the remainder of the graph, and since the only edges within $A_i$ come from $G(n,p)$, we know that they are $K_t$-free.  Any red component within $B$ that corresponds to a clique in $\varphi$ can involve at most $a$ parts $B_i$.  The largest clique within such a part has size at most $\ell - 1$, and so the largest red clique in $B$ has size at most $a(\ell - 1) < t$.  Hence there is indeed no red $K_t$.

Within the blue subgraph, the parts $A_i$ and $B_i$ are independent sets, and hence any blue clique $K$ can contain at most one vertex from each part.  Moreover, by the colouring $\varphi$, there can be at most $s-k-1$ vertices from $B$ in $K$.  As there are only $k$ parts in $A$, this shows that the largest blue clique has size at most $s-1$, and hence there is no blue $K_s$ either.  This completes the proof of the lower bound.
\end{proof}

We next establish the upper bound $p(n; K_t, K_s, d) \le n^{-2/t}$.

\begin{proof}[Proof of Proposition~\ref{prop:K4vsclique} (upper bound)]
We start by defining sequences of constants that we shall require in the proof.  Let $d := 1 - 1/(k-1) + \delta$ for some $\delta > 0$.  We then set $\delta_1 := \delta / 2$ and shall soon (implicitly) specify a sufficiently small $\eps_1$.  Let $\eta = \eta(\eps_1,\delta _1,  1)$ be as given by Corollary~\ref{cor:regtuple}, and set $m_1 := \eta n$.  Now, for $i \in [k-1]$, set $\delta_{i+1} := \delta_i / 2$, $m_{i+1} := \delta_i^2 m_i / 4$, and $\eps_{i+1} := 4 \eps_i / \delta_i^2$.  Observe that each $m_i$ is linear in $n$, and that the ratios $m_i/m_1$ are independent of $\eps_1$.

Given these constants, choose $\alpha = \alpha(\delta)  > 0$ sufficiently small to ensure $2 k \alpha m_1 \binom{m_i}{t-1} \le \frac14 \binom{m_i}{t}$ for all $i$, noting that this is independent of our choice of $\eps_1$.  Finally, we choose $\eps_1$ to be small enough that $\eps_1 < \alpha$, $\delta_i \ge 3 \eps_i$ for each $i$, and $2 k \eps_i m_i^2 \binom{m_i}{t-2} \le \frac14 \binom{m_i}{t}$ for all $i$.  With these preliminaries sorted, we can now begin the proof.

\medskip

We first seek a well-structured part of the deterministic graph $G$ before exposing the random edges.  Since $G$ has density $d$, we can apply Corollary~\ref{cor:regtuple} with $r = 1$ to find sets $V_1, V_2. \hdots, V_k$, each of size 
$m_1$, such that each pair $(V_i, V_j)$ is $\eps_1$-regular of density at least $\delta_1$.

Note that this is taking place in the uncoloured graph $G$, and hence the parts $V_i$ are determined before we expose the random edges of $G(n,p)$.  The edges in these regular pairs could later be coloured either red or blue, but Claim~\ref{clm:blueedges} will show that we may assume they are almost all blue.  For this, we require the following properties of our random graph $G(n,p)$, where $\beta = \beta(t,s)$ is defined below:
\begin{itemize}
	\item[(a)] $G(n,p)$ is $(\delta_1 \alpha \eta / 4)$-globally $(K_{t-1}, K_{s-1})$-Ramsey, and
	\item[(b)] if $t > s$, $G(n,p)$ is $n^{-\beta}$-globally $(K_{\ceil{t/2}}, K_s)$-Ramsey.
\end{itemize}

For the first property, we have $m_2(K_{t-1}, K_{s-1}) \le m_2(K_{t-1}) = t/2$.  Next, if $t > s$, let $s' := \min ( s, \ceil{t/2})$ and $t' := \max (s, \ceil{t/2})$.  Observe that $m_2(K_{t'}, K_{s'}) < t/2$, and so we can find some $\beta = \beta(t,s) > 0$ such that $m_2(K_{t'}, K_{s'})/(1-\beta) = t/2$.  By Theorem~\ref{thm:global1ramsey}, it follows that if $p = Cn^{-2/t}$ for some suitably large constant $C$, then $G(n,p)$ has both properties with high probability.

\begin{claim} \label{clm:blueedges}
If $G(n,p)$ has Properties (a) and (b) above, then we may assume that in any $2$-colouring of $G \cup G(n,p)$ with neither a red $K_t$ nor a blue $K_s$, for every pair $i \neq j$, the maximum red-degree in $(V_i, V_j)$ is at most $\alpha m_1$.
\end{claim}

We shall prove Claim~\ref{clm:blueedges} in due course, but let us first see how it implies our desired upper bound on $p(n; K_t, K_s, d)$.  Roughly speaking, we shall find $t$-cliques within each part $V_i$.  Since there are no red $t$-cliques, each such $t$-clique must contain a blue edge.  We shall choose the cliques to ensure that these blue edges combine to form a blue $K_{2k}$, contradicting our colouring being blue-$K_s$-free.

More precisely, let $p=C n^{-2/t}$ where $C$ is a sufficiently large constant, and assume (a) and (b) above hold.  Consider any $2$-colouring of $G \cup G(n,p)$, and recall that we assume the vertex set to be $[n]$, which we equip with its natural ordering.

Suppose for some $a \ge 0$, we have selected a set $S_a = \{s_1, s_2, \hdots, s_{2a} \}$ of vertices from $\cup_{i = 1}^a V_i$, such that they induce a blue $K_{2a}$ in our colouring of $G \cup G(n,p)$ and have at least $m_{a+1}$ common neighbours in each $V_i$ for $a+1 \le i \le k$.  When $a = 0$, the set $S_0 := \emptyset$ trivially satisfies these requirements.

For each $a+1 \le i \le k$, let $\vSet{i}{a+1}$ be the first $m_{a+1}$ common neighbours of $S_a$ in $V_i$.  Our goal is to find two vertices $s_{2a+1}, s_{2a+2} \in \vSet{a+1}{a+1}$ to add to $S_a$ in order to obtain a valid set $S_{a+1}$ with which to proceed.  In particular, the two new vertices should share a blue edge, all edges from them to $S_a$ should also be blue, and $s_{2a+1}$ and $s_{2a+2}$ should have many common neighbours in each of the remaining parts.  We therefore define $\mc H_{a+1}$ to be the collection of all copies of $K_t$ whose vertex sets $T \subseteq \vSet{a+1}{a+1}$ satisfy the following properties:
\begin{itemize}
	\item[(A)] all edges between $S_a$ and $T$ are blue, and
	\item[(B)] every pair of vertices in $T$ have at least $m_{a+2}$ common neighbours in $\vSet{j}{a+1}$ for each $a+2 \le j \le k$.
\end{itemize}

Let $R \subset \vSet{a+1}{a+1}$ be those vertices that have a red edge to some vertex in $S_a$.  By Claim~\ref{clm:blueedges}, we have $\card{R} \le \alpha m_1 \card{S_a} \le 2k \alpha m_1$.  Since every $t$-clique in $\vSet{a+1}{a+1}$ violating Condition (A) must contain a vertex from $R$, there are at most $2k \alpha m_1 \binom{m_{a+1}}{t-1}$ such cliques.  By our choice of $\alpha$, this is at most $\frac14 \binom{m_{a+1}}{t}$ cliques.

By Lemma~\ref{lem:regpairs}(ii), it follows that, for each $a+2 \le j \le k$, the pair $(\vSet{a+1}{a+1}, \vSet{j}{a+1})$ is $\eps_{a+1}$-regular of density at least $\delta_{a+1}$.  We may therefore apply Lemma~\ref{lem:regpairs}(i) to deduce that, for each given $j$, there are at most $2 \eps_{a+1} m_{a+1}^2$ pairs in $\vSet{a+1}{a+1}$ with fewer than $m_{a+2}$ common neighbours in $\vSet{j}{a+1}$.  There are thus a total of at most $2k \eps_{a+1} m_{a+1}^2 \binom{m_{a+1}}{t-2}$ cliques of size $t$ in $\vSet{a+1}{a+1}$ violating Condition (B).  Our choice of $\eps_1$ ensures that this is again at most $\frac14 \binom{m_{a+1}}{t}$ cliques.

\medskip

Thus $\mc H_{a+1}$ contains at least $\frac12 \binom{m_{a+1}}{t}$ sets.  Suppose we find in $G(n,p)$ a $t$-clique $H \in \mc H_{a+1}$.  If all edges of $H$ are red, we have our desired red $K_t$, and hence we may assume there is some blue edge $e_{a+1} = \{s_{2a+1}, s_{2a+2} \}$ in $H$.  

This pair of vertices has all the properties we required: they share the blue edge $e_{a+1}$, Property (A) of $T$ ensures that all edges from $e_{a+1}$ to $S_a$ are blue, and Property (B) gives that, for each $j \ge a+2$, $s_{2a+1}$ and $s_{2a+2}$ have at least $m_{a+2}$ common neighbours in $\vSet{j}{a+1}$.

Hence, provided we can find a $t$-clique from the collection $\mc H_{a+1}$, we may set $S_{a+1} := S_a \cup \{s_{2a+1}, s_{2a+2} \}$ and proceed to the next iteration.  We appeal to Theorem~\ref{thm:janson} to find the desired clique.

\medskip

Indeed, we know $\card{\mc H_{a+1}} \ge \frac12 \binom{m_{a+1}}{t}$, which, since $m_{a+1}$ is linear in $n$, is at least $\xi n^t$ for some constant $\xi = \xi(\delta) > 0$.  Moreover, when $H := K_t$, we have $\mu _1 =\mu_1 (H)= n^t p^{\binom{t}{2}}$.  Recall $p = C n^{-2/t}$, so $\mu_1 = C ^{\binom{t}{2}} n$.  Theorem~\ref{thm:janson} thus gives that the probability $G(n,p)$ does not contain a $t$-clique from a given collection $\mc H_{a+1}$ is at most $\exp( - C' n)$, where $C' := \xi C^{\binom{t}{2}} / (2^{t+1} t!)$.

However, the collection of $t$-cliques $\mc H_{a+1}$ depends on the colouring of the edges of $G$, which in turn could depend on the random graph $G(n,p)$ itself.  To resolve this issue, we take a union bound over all collections of $t$-cliques $\mc H_{a+1}$ that could possibly arise.  Observe that the sets $\vSet{i}{a+1}$, for $a+1 \le i \le k$, are determined by the set $S_a$.  This already specifies which cliques in $\vSet{a+1}{a+1}$ fail to satisfy Property (B).  To identify those violating Property (A), it suffices to identify the subset $R \subset \vSet{a+1}{a+1}$ of vertices incident to a red edge from $S_a$.

The collection $\mc H_{a+1}$ is thus fully determined by the pair $(S_a, R)$, and there are fewer than $n^{2a} 2^{m_{a+1}}$ such pairs, which we can (wastefully) bound from above by $4^n$.  Hence, as $C$ has been chosen sufficiently large (with respect to $\xi$ and $t$), it follows from the union bound that with high probability, for each $0 \le a \le k-1$ and for every possible collection $\mc H_{a+1}$ that may arise, $G(n,p)$ contains a copy of $K_t$ from $\mc H_{a+1}$.

\medskip

We can thus repeat this process until we obtain a set $S_k$ that induces a blue $K_{2k}$.  Since $s \le 2k$, this shows that with high probability $G \cup G(n,p)$ is indeed $(K_t, K_s)$-Ramsey, as desired.
\end{proof}

It remains to prove Claim~\ref{clm:blueedges}.

\begin{proof}[Proof of Claim~\ref{clm:blueedges}]

We begin with a straightforward observation: we cannot have a vertex $u$, a vertex $v$, and a set $U$ of $\delta_1 \alpha m_1 / 4$ common neighbours of $u$ and $v$, such that all edges from $u$ to $U$ are red and all edges from $v$ to $U$ are blue.  Indeed, by Property $(a)$, there is  a red $K_{t-1}$ or a blue $K_{s-1}$ in $U$.  Extending this monochromatic clique by $u$ or $v$ respectively gives either a red $K_t$ or a blue $K_s$, contradicting our assumption on the colouring of $G \cup G(n,p)$.

\medskip

Let us first consider the case $t > s$.  Suppose for contradiction we have some $u \in V_i$ with a set $R$ of $\alpha m_1$ red neighbours in $V_j$, for some $j \neq i$.  Let $S := \{ v \in V_i : |N_G (v)\cap R| \ge (\delta_1 - \eps_1) \card{R} \}$.  By the $\eps_1$-regularity of $(V_i, V_j)$, it follows that $\card{S} \ge (1 - \eps_1) m_1$.

By our earlier observation, each $v \in S$ can have at most $\delta_1 \alpha m_1 / 4$ blue edges to $R$, and hence there are at least $(1 - \eps_1)(\delta_1 - \eps_1 - \delta_1 / 4) \card{R} m_1 \ge \delta_1 \card{R} m_1 / 8$ red edges between $S$ and $R$.  It follows from dependent random choice (see~\cite{fs}) that we can find a subset $U \subset S$ of size $n^{1- \beta}$, such that every subset of $\ceil{t/2}$ vertices from $U$ has at least $n^{1 - \beta}$ common red neighbours in $R$.

By Property (b), we find a red $K_{\ceil{t/2}}$ or a blue $K_s$ in $U$.  In the latter case, we are done, so we may assume the former.  Let $A$ be the set of vertices of this $\ceil{t/2}$-clique, and let $W\subseteq R$ be a set of $n^{1 - \beta}$ common red neighbours of $A$.  Applying Property (b) to $W$, we again  find a blue $K_s$, and are done, or find a set $B$ of $\ceil{t/2}$ vertices inducing a red clique.  In this latter case, $A \cup B$ gives rise to a red clique on at least $t$ vertices, and hence we have the desired contradiction.

\medskip

This leaves us with the case $t = s$, where by symmetry we may assume that the majority of edges in $G \cup G(n,p)$ are coloured blue.  For a contradiction, we suppose without loss of generality that there is some $u \in V_1$ with a set $R$ of $\alpha m_1$ red neighbours in $V_2$.  Since we assumed blue was the more popular colour, we can find some (ordered) pair $(V_i, V_j)$, $j \neq 2$, where at least one-third of the edges are blue.  By averaging, this gives some vertex $v \in V_i$ and a set $S \subseteq V_j$ of $\alpha m_1$ blue neighbours of $v$.

Now note that $(V_2, V_j)$ is an $\eps_1$-regular pair, and so there are at least $(\delta_1 - \eps_1) \alpha^2 m_1^2$ edges between $R$ and $S$.  If at least half of these edges were coloured red, then by averaging, we would find a vertex $u' \in R$ with a set $U$ of at least $\delta_1 \alpha m_1 / 4$ red neighbours in $S$.  Then the vertices $u'$ and $v$, together with the set $U$, violate our initial observation.  On the other hand, if half of the edges between $R$ and $S$ are blue, then we find a vertex $v' \in S$ with a set $U'$ of at least $\delta_1 \alpha m_1 / 4$ blue neighbours in $R$.  Then $u, v'$ and $U'$ violate our initial observation instead.

Hence, if the colouring of $G \cup G(n,p)$ has neither a red $K_t$ nor a blue $K_s$, we may indeed assume that the maximum red-degree in each pair $(V_i, V_j)$ is at most $\alpha m_1$.
\end{proof}


\subsection{Proof of Theorem~\ref{thm:twocycles}}
First we restate Theorem~\ref{thm:twocycles} in the following equivalent form.
\begin{thm}\label{c1}
 Let $\ell, k \geq 3$ be integers.
\begin{itemize}
\item[(i)] If $k,\ell \in 2 \mathbb N$ then $p(n; C_k,C_{\ell}, d )=0$ for all $d >0$.
\item[(ii)] If $k \in 2 \mathbb N+1$  then $p(n; C_k,C_{\ell}, d )=1/n$ for all $d \in (0,1/2]$.
\item[(iii)] If $k \in 2 \mathbb N+1$ and $\ell \in 2 \mathbb N$ then $p(n; C_k,C_{\ell}, d )=0$ for all $d >1/2$.
\item[(iv)] 
\[ p(n; C_3, C_{3}, d) = \begin{cases}
	1/n^2 & \textrm{if } d \in (1/2,4/5] \\
0 & \textrm{if } d > 4/5.
\end{cases} \]
\item[(v)] If $k \in 2 \mathbb N+1$ and $\ell \in 2 \mathbb N+3$ then 
\[ p(n; C_k, C_{\ell}, d) = \begin{cases}
	1/n^2 & \textrm{if } d \in (1/2,3/4] \\
0 & \textrm{if } d > 3/4.
\end{cases} \]
\end{itemize}
\end{thm}
\subsubsection{Proof of  Theorem~\ref{c1}(i)}
 Let $r:=R_{bip}(k,\ell)$, the bipartite Ramsey number for $K_{k,k}$ and $K_{\ell,\ell}$.  Then $K_{r,r}$ is a bipartite graph that is $(K_{k,k},K_{\ell,\ell})$-Ramsey, and therefore $(C_k, C_{\ell})$-Ramsey as well.  By Observation~\ref{obs:trivial}(iii), it follows that $p(n;C_k, C_{\ell},d) = 0$ for $d > 0$; that is, any sufficiently large dense graph will already be $(C_k, C_{\ell})$-Ramsey before any random edges are added.
\qed

\subsubsection{Proof of  Theorem~\ref{c1}(ii)}
The lower bound here is the trivial lower bound from Observation~\ref{obs:trivial}, since $C_k$ is not bipartite and for $p = o(1/n)$, $G(n,p)$ is with high probability $C_{\ell}$-free.  Thus $p(n;C_k, C_{\ell}, d) \ge 1/n$ for all $d \in (0,1/2]$.

\smallskip

Let $d>0$ and set $0<\eta \ll \eps \ll \gamma \ll \delta \ll d,1/k,1/\ell$. Suppose that $G$ is a sufficiently large $n$-vertex graph. Consider $G\cup G(n,p)$ where $p=\omega (1/n)$.
By Corollary~\ref{cor:regtuple} (with $r=1$) there exist  disjoint $A,B\subseteq V(G)$ so that $|A|=|B|\geq \eta n$ and $(A,B)_{G}$ is $\eps$-regular with density at least $\delta$.
We call an \emph{ordered} $(k+\ell)$-tuple $(x_1,\dots,x_k,y_1,\dots, y_{\ell}) \in A^{k+\ell}$ \emph{good} if  $|(\cap _i N_{G}(x_i)) \cap  (\cap _i N_{G}(y_i))\cap B|\geq \gamma |B|$.
Lemma~\ref{lem:regpairs}(i) implies that all but at most $\gamma |A|^{k+\ell}$ ordered $(k+\ell)$-tuples in $A^{k+\ell}$ are good.

Using Theorem~\ref{thm:janson}, we deduce that with high probability there is a good $(k+\ell)$-tuple such that in $G(n,p)$, $x_1 x_2 \hdots x_k x_1$ forms a $k$-cycle and $y_1 y_2 \hdots y_{\ell} y_1$ gives an $\ell$-cycle.  Indeed, letting $H$ be the vertex-disjoint union of $C_k$ and $C_{\ell}$, our above argument shows that the collection $\mc H$ of potential copies of $H$ with good supports has size $\xi n^{k + \ell}$ for some $\xi > 0$.  Moreover, we have $\mu_1 = (np)^{\min(k,\ell)} = \omega(1)$, and so the probability of having no copy of $H$ supported on a good $(k + \ell)$-tuple tends to zero as $n$ grows. Therefore w.h.p. in $G(n,p)[A]$ we have a disjoint $k$-cycle $X$  and $\ell$-cycle $Y$ together with a set $N\subseteq B$ where $|N| \geq \gamma |B|$ and $N\subseteq N_G (X \cup Y)$.

Consider any $2$-colouring of $G \cup G(n,p)$. Given any $b \in B$ 
there are $2^{k+\ell}$ possible ways to colour those edges incident to $b$ with an endpoint in $X\cup Y$. Thus, there is  
a set $N_1 \subseteq N$ of size at least $\gamma |B|/2^{k+\ell}$ where each vertex in $N_1$ has the same `colour profile' (i.e. every vertex in $N_1$ has the same red neighbourhood in $X\cup Y$ and therefore the same blue 
neighbourhood in $X\cup Y$).

Colour each vertex $v \in X \cup Y$ with the colour of the edges it receives from $N_1$. So we now have a red/blue colouring of the vertices and edges of $X \cup Y$.
Suppose there is a red vertex $u \in X \cup Y$ and a blue vertex $v \in X \cup Y$.
Theorem~\ref{thm:global1ramsey}(ii) implies that w.h.p. $G(n,p)[N_1]$ is $(P_{k-1},P_{\ell-1})$-Ramsey.
If there is a red copy of $P_{k-1}$ in $N_1$, then together with $u$ we obtain our desired red copy of $C_k$ in $G\cup G(n,p)$;
otherwise we obtain a blue copy of $P_{\ell-1}$ in $N_1$ and thus together with $v$ we obtain a blue copy of $C_{\ell}$.

Thus, we may assume that every vertex in $X\cup Y$ is coloured the same. Suppose that they are all red. If the edges of $Y$ are all blue, we obtain the desired blue copy of $C_{\ell}$.
So we may assume that there is at least one red edge in $Y$. As all the edges between $X\cup Y$ and $N_1$ are red, we can extend this red edge to a red copy of $C_k$ in $G \cup G(n,p)$.
The case when every vertex in $X\cup Y$ is blue is similar (and in fact easier if $\ell$ is even), and so in all cases we obtain a desired monochromatic cycle.
\qed

\subsubsection{Proof of  Theorem~\ref{c1}(iii)}
Let $d>1/2$ and define $\delta >0$ so that $d>1/2 +\delta$. Set $0< \eps \ll \delta$. Let $\eta>0$ be obtained by applying Corollary~\ref{cor:regtuple} with input $r=2$.
Consider any sufficiently large $n$-vertex graph $G$ of density at least $d$, and consider any $2$-colouring of $G$.
By Corollary~\ref{cor:regtuple} we have that there are disjoint sets $V_1,V_2,V_3$ in $G$ so that $|V_1|=|V_2|=|V_2|\geq \eta n$ and for each $1 \leq i <j \leq 3$,
there is some colour $c_{i,j} $ for which the edges between $V_i$ and $V_j$ of colour $c_{i,j}$ form an $\eps$-regular pair of density at least $\delta/4$.
Suppose one of these colours $c_{i,j}$ is blue. Then Lemma~\ref{keylem} implies $G$ contains a blue copy of $C_\ell$.
Otherwise all the $c_{i,j}$ are red and then since $C_{k}$ is $3$-partite, Lemma~\ref{keylem} implies $G$ contains a red copy of $C_{k}$.
\qed

\subsubsection{Proof of  Theorem~\ref{c1}(iv)}
Let $G_n$ denote the $5$-partite Tur\'an graph on $n$ vertices. Since $K_5$ has a $2$-colouring without a monochromatic copy of $C_3$, so does $G_n$.
Let $p=o(1/n^2)$. Then w.h.p. $G(n,p)$ is empty. Thus, w.h.p. $G_n \cup G(n,p)$ is not $(C_3,C_3)$-Ramsey. 
This shows that $p(n, C_3,C_3;d)\geq 1/n^2$ for all $d \leq 4/5$.

Next suppose that $d>1/2$, and let $p=\omega(1/n^2)$. Since $K_6$ is $(C_3,C_3)$-Ramsey, to prove that $p(n; C_3,C_3, d)= 1/n^2$, it suffices to show that given any $n$-vertex graph $G$ of density $d$, w.h.p.
$K_6 \subseteq G\cup G(n,p)$.  This follows immediately from Theorem~\ref{thm:perturbedturan} with $k = 3$, since $\rho_3(K_6) = 1/2$.  So indeed $p(n; C_3,C_3, d)= 1/n^2$.

\smallskip

Let $d>4/5$ and suppose that $G$ is any sufficiently large graph  with density at least $d$. Then by Tur\'an's theorem $G$ contains a copy of $K_6$. Since $K_6$ is $(C_3,C_3)$-Ramsey, any $2$-colouring of $G$ yields a monochromatic
copy of $C_3$. Thus, $p(n; C_3,C_{3}, d)=0$.
\qed

\subsubsection{Proof of  Theorem~\ref{c1}(v)}
To prove this result we will need two additional lemmas.
\begin{lemma}\label{l1}
$R( \{C_3\}, \{C_3,C_5\})=5.$
\end{lemma}
\proof
To see that $R(\{C_3\},\{C_3,C_5\})>4$, 
consider a $2$-colouring of $K_4$ whose red edges induce a copy of $P_4$ and whose blue edges induce a copy of $P_4$. 

Next consider any $2$-colouring of $K_5$. If there is a vertex  incident to at least three red edges then we must have a red or blue copy of $C_3$.
The same conclusion holds if there is a vertex incident to at least three blue edges. Thus, we may assume that every vertex has red and blue degree two. In particular,
the blue subgraph is $2$-regular and so is a copy of $C_5$, as desired.
\endproof
Let $H_m$ be the graph formed by taking disjoint vertex sets $V_1,\dots,V_5$ each of size $m$, and with edge set as follows:
$H_m$ contains a perfect matching between $V_1$ and $V_2$;  a perfect matching between $V_3$ and $V_4$; between all other pairs of distinct $V_i$ there are all possible edges.
The next result proves that finding $H_m$ in a graph $G$ (for $m$ sufficiently large) ensures $G$ is $(C_k,C_{\ell})$-Ramsey.
\begin{lemma}\label{l2}
Given  $k \in 2 \mathbb N+1$ and $\ell \in 2 \mathbb N+3$, there exists an $m_0=m_0(k,\ell)$ such that if $m \geq m_0$, then $H_m$ is $(C_k,C_{\ell})$-Ramsey.
\end{lemma}
\proof
Consider any $2$-colouring $c$ of $H_m$. Given an edge $xy \in E(H_m)$, write $c(xy)$ for the colour of $xy$ in $c$.
We first build an auxiliary complete bipartite graph $B$ with classes $B_1$ and $B_2$.
The vertices of $B_1$ are the edges $\{v_{1,i}v_{2,i}  : i \in [m] \}$ of the perfect matching between $V_1$ and $V_2$ in $H_m$;
the vertices of $B_2$ are the edges $\{v_{3,i}v_{4,i}  : i \in [m] \}$ of the perfect matching between $V_3$ and $V_4$ in $H_m$.

Next we $16$-colour the edges of $B$ as follows: for all $i,j \in [m]$, we colour the edge from $v_{1,i}v_{2,i}$ to $v_{3,j}v_{4,j}$ in $B$ 
with the $4$-tuple $(c(v_{1,i}v_{3,j}), c(v_{1,i}v_{4,j}), c(v_{2,i}v_{3,j}), c(v_{2,i}v_{4,j}) )$.
Since $m$ is sufficiently large, the bipartite Ramsey theorem implies the existence of a monochromatic copy $K$ of $K_{4(k+\ell), 4(k+\ell)}$ in $B$.
Let $(c_{1,3}, c_{1,4}, c_{2,3}, c_{2,4})$ be the colour of $K$. Let $V'_i \subseteq V_i$ denote the set of vertices in $V_i$ that are `present' in $K$; for example,
$v_{1,j} \in V'_1$ precisely if $v_{1,j}v_{2,j}$ is a vertex in $K$.
It follows that, for every $i \in \{1,2\}$, $j \in \{3,4\}$, all edges in the complete bipartite graphs
$H_m[V'_i,V'_j]$ have the colour $c_{i,j}$.

Now consider the vertices in $V_5$. There are $2^{16{(k+\ell)}}$ possible ways the edges between a vertex $v \in V_5$ and the vertices in $\cup ^4 _{i=1} V'_i$ can be coloured.
Hence, we can find a set $V''_5 \subseteq V_5$ of at least $m 2^{-16(k+\ell)} \geq k+\ell$ vertices that all have the same colour profile.

Next consider the matching between $V'_1$ and $V'_2$ in $H_m$. For each edge $v_1v_2$ in this matching, there are four possible ways the edges from $\{v_1,v_2\}$ to $V''_5$ can be coloured. 
Hence, there are subsets $V''_1 \subseteq V'_1$ and $V''_2 \subseteq V'_2$ of size $k+\ell$ such that there is a perfect matching in $H_m[V''_1,V''_2]$, and for each $i \in [2]$, all edges
between $V''_i$ and $V''_5$ have colour $c_{i,5}$. Similarly, there are subsets $V''_3 \subseteq V'_3$ and $V''_4 \subseteq V'_4$ of size $k+\ell$ such that there is a perfect matching in $H_m[V''_3,V''_4]$, and for each $i \in \{3,4\}$, all edges
between $V''_i$ and $V''_5$ have colour $c_{i,5}$.

In summary, we have an induced subgraph $H'_m$ of $H_m$ on $\cup _{i=1} ^5 V''_i$ where $|V_i ''|\geq k+\ell$; a perfect matching in $H'_m[V''_1,V''_2]$ and in $H'_m[V''_3,V''_4]$; 
for all other pairs $i<j$, a monochromatic complete bipartite graph of colour $c_{i,j}$ between $V''_i$ and $V''_j$ in $H'_m$.

Fix an edge $x_1x_2$ in $H'_m[V''_1,V''_2]$ and denote its colour in $c$ by $c_{1,2}$. Similarly let $x_3x_4$ be an edge in $H'_m[V''_3,V''_4]$ with colour $c_{3,4}$.
Now consider an auxiliary copy of $K_5$ with vertex set $[5]$, colouring each edge $ij$ with colour $c_{i,j}$. By Lemma~\ref{l1} we find a red $C_3$, blue $C_3$ or blue $C_5$.

{\noindent \bf Case 1: There is a monochromatic copy of $C_3$ on $\{i,j,r\}$ in $K_5$.}
Suppose this $C_3$ is red; the blue case is analogous. Choose vertices $y_1 \in V''_i$, $y_2,y_4,\dots, y_{k-1} \in V''_j$ and $y_3,y_5, \dots , y_k \in V''_r$.
These choices can be made arbitrarily unless $\{1,2\} \subseteq \{i,j,r\}$ in which case we set $i:=1$, $j:=2$, $y_1:=x_1$ and $y_2:=x_2$; or if
$\{3,4\} \subseteq \{i,j,r\}$, in which case we set $i:=3$, $j:=4$, $y_1:=x_3$ and $y_2:=x_4$.

Note that $y_1y_2y_3\dots y_ky_1$ is a red copy of $C_k$ in $H'_m \subseteq H_m$, as desired.

{\noindent \bf Case 2: There is a blue  copy  $\pi _1 \pi_2 \pi_3 \pi_4 \pi_5 \pi _1$ of  $C_5$  in $K_5$.}
We may assume that $\pi _5=5$. Choose vertices $y_1 \in V''_{\pi _1}$, $y_2 \in V''_{\pi _2}$, $y_3 \in V''_{\pi _3}$, $y_4,y_6, \dots, y_{\ell-1} \in V''_{\pi _4}$
and $y_5,y_7, \dots, y_{\ell} \in V''_{5}$ such that $y_i = x_{\pi _i} \in V''_{\pi _i}$ for each $i \in [4]$, with the other vertices chosen arbitrarily.
Then $y_1y_2y_3\dots y_\ell y_1$ is a blue copy of $C_\ell$ in $H'_m \subseteq H_m$, as desired.
\endproof

With Lemma~\ref{l2} at hand it is now straightforward to prove Theorem~\ref{c1}(v).

{\noindent \it Proof of Theorem~\ref{c1}(v)}
 Let $G_n$ denote the $4$-partite Tur\'an graph on $n$ vertices. Colour the edges of $K_4$ so that the red edges induce a $P_4$, and the blue edges also a $P_4$. Lift this to a $2$-colouring of $G_n$.
Since the red (blue) subgraph of $G_n$ is bipartite we do not have a red (blue) copy of any odd cycle. 
Let $p=o(1/n^2)$. Then w.h.p. $G(n,p)$ is empty. Thus, w.h.p. $G_n \cup G(n,p)$ is not $(C_k,C_\ell)$-Ramsey. 
This shows that $p(n; C_k,C_\ell,d)\geq 1/n^2$ for all $d \leq 3/4$.

Next suppose that $d>1/2$. Let $p=\omega (1/n^2)$, and let $m\geq m_0$ as in Lemma~\ref{l2}.
Let $G$ be a sufficiently large $n$-vertex graph of density at least $d$. The $3$-partition $V(H_m) = (V_1 \cup V_2) \cup (V_3 \cup V_4) \cup V_5$ shows $\rho_3(H_m) \le 1/2$, and so by Theorem~\ref{thm:perturbedturan}, w.h.p. $H_m \subseteq G \cup G(n,p)$, and therefore $G \cup G(n,p)$ is $(C_k, C_{\ell})$-Ramsey.  So indeed $p(n; C_k,C_\ell,d)= 1/n^2$ for all $d \in (1/2,3/4]$.

Finally, suppose that $d>3/4$ and let $G$ be a sufficiently large graph of density $d$. Since $H_m$ is $5$-partite, the Erd\H{o}s--Stone--Simonovits theorem~\cite{erdosstone} implies that $H_m \subseteq G$. Lemma~\ref{l2} implies that
$G$ is $(C_k,C_{\ell})$-Ramsey. So $p(n; C_k,C_\ell,d)= 0$ for all $d > 3/4$.
\qed


\subsection{Proof of Theorem~\ref{thm:manycycles}}

We next turn to our multicolour result, where we seek the threshold at which any $r$-colouring of $G \cup G(n,p)$ will contain a monochromatic copy of $C_{\ell}$, where $\ell \ge 2^r + 1$ is odd.
  In the proof of the theorem we will repeatedly make use of the following simple property of $2^i$-partite graphs.
\begin{fact} \label{fact2}
The edge set of any $2^i$-partite graph $H$ can be partitioned into $i$ bipartite graphs.
\end{fact}
The $2^i$ is best possible, as shown by the following Ramsey result, which can be proven by induction on $r$.
\begin{fact}\label{fact1}
Given any $r \geq 1$, and any $r$-colouring of $K_{2^r+1}$, there is a monochromatic odd cycle in $K_{2^r+1}$.
\end{fact}

\subsubsection{The case when $0\leq d \leq 1-2^{-r+2}$.}

When the underlying graph $G$ is not very dense, we can appeal to Observation~\ref{obs:multicol}.  By Fact~\ref{fact2}, any graph that is $(C_{\ell},r-2)$-Ramsey must have chromatic number at least $2^{r-2} + 1$, as otherwise its edges can be partitioned into $r-2$ bipartite (and hence $C_{\ell}$-free) subgraphs.  Thus, by Observation~\ref{obs:multicol}, if $d \le 1 - 2^{-r+2}$, we have $p(n;r,C_{\ell},d) = n^{-1/m_2(C_{\ell})} = n^{-1 + 1/(\ell - 1)}$.

\smallskip
\subsubsection{The case when $1-2^{-r+2}< d \leq 1-2^{-r+1}$.}

For each $n \in \mathbb N$, let $G_n$ denote the $2^{r-1}$-partite Tur\'an graph on $n$-vertices.
Let $p = o(n^{-1})$. By Fact~\ref{fact2} we can $(r-1)$-colour $G_n$ so that each colour class is bipartite, thus avoiding monochromatic copies of $C_{\ell}$.
A simple application of Markov's inequality yields that w.h.p. $G(n,p)$ does not contain a copy of $C_{\ell}$, and thus its edges can be coloured with the remaining colour.
Together this implies that w.h.p. $G_n \cup G(n,p)$ is not $(C_\ell,r)$-Ramsey. Hence, $p(n;r,C_{\ell},d) \ge n^{-1}$ for all $d \leq 1-2^{-r+1}$.  As in the previous case, the upper bound follows from Theorem~\ref{randomramsey}, since $n^{-1 + 1/(\ell - 1)}$ is the threshold for $G(n,p)$ itself to be $(C_{\ell},r)$-Ramsey.  This shows $p(n; r, C_{\ell}, d) \in [n^{-1}, n^{-1 + 1/(\ell - 1)}]$ for all $1 - 2^{-r + 2} < d \le 1 - 2^{-r+1}$.

\smallskip
\subsubsection{The case when $1-2^{-r+1}< d \leq 1-2^{-r}$.}

For each $n \in \mathbb N$, let $G_n$ denote the $2^{r}$-partite Tur\'an graph on $n$-vertices.
Let $p=o(n^{-2})$. By Fact~\ref{fact2} we can $r$-colour $G_n$ so that  there are no monochromatic copies of $C_{\ell}$.
Further w.h.p. $G(n,p)$ is empty.
So w.h.p. $G_n \cup G(n,p)$ is not $(C_\ell,r)$-Ramsey. Hence, $p(n;r,C_{\ell},d)\geq n^{-2}$ for all $1-2^{-r+1}< d \leq 1-2^{-r}$.

To finish this case off, we require the following generalisation of the graph $H_m$ introduced in the proof of Theorem~\ref{thm:twocycles}.
Let $H_{m,r}$ be the graph formed by taking disjoint vertex sets $V_1,\dots,V_{2^r+1}$ each of size $m$, and with edge set as follows:
$H_{m,r}$ contains a perfect matching between $V_{2i-1}$ and $V_{2i}$ for each $1\leq i \leq 2^{r-1}$;  between all other  distinct pairs $V_i$, $V_j$ there are all possible edges.
Note that $H_m=H_{m,2}$.
\begin{lemma}\label{l3}
Given   $\ell \geq 2^{r}+1$, there exists an $m_0=m_0(r,\ell)$ such that if $m \geq m_0$, then $H_{m,r}$ is $(C_{\ell},r)$-Ramsey.
\end{lemma}
\proof
Let $m$ be sufficiently large and
consider any $r$-colouring $c$ of $H_{m,r}$.
By repeatedly arguing as in the proof of Lemma~\ref{l2} we can obtain subsets $V^*_i \subseteq V_i$ (for each $ i \leq 2^r$) of size $r^3\ell$ so that:
\begin{itemize}
\item There is a perfect matching from $V^*_{2i-1}$ to $V^*_{2i}$ for every $1\leq i \leq 2^{r-1}$;
\item For all other pairs  $(V^*_i,V^*_j)$ (with $ i <j \leq 2^r$),  all  edges from $V^*_i$ to $V^*_j$ have the same colour $c_{i,j}$.
\end{itemize}
To construct such sets $V^*_i$ we first construct the auxiliary bipartite graph $B$ precisely as in Lemma~\ref{l2};
this yields sets $V'_1,\dots, V'_4$ (whose sizes will now still be huge).
 Then we construct an analogous auxiliary graph for the pair $((V'_1,V'_2), (V_5,V_6))$, and so on, until we have considered every pair of tuples of such vertex classes.

By further restricting to the most popular colour within each of the perfect matchings, we can now assume that each such set has size $r^2\ell$ and 
\begin{itemize}
\item There is a \emph{monochromatic} perfect matching from $V^*_{2i-1}$ to $V^*_{2i}$ for every $1\leq i \leq 2^{r-1}$. Denote this colour by $c_{2i-1,2i}$.
\end{itemize}

Now consider the vertices in $V_{2^r+1}$. There are $r^{r^2\ell 2^r}$ possible ways the edges between a vertex $v \in V_{2^r}+1$ and the vertices in $\cup ^{2^r} _{i=1} V^*_i$ can be coloured.
Hence, we can find a set $V^{**}_{2^r+1} \subseteq V_{2^r +1}$ of size at least $m r^{-r^2\ell2^r} \geq \ell$ vertices that all have the same colour profile.

Next consider the matching between $V^*_{2i-1}$ and $V^*_{2i}$ in $H_{m,r}$ for all $1\leq i \leq 2^{r-1}$. For each edge $v_{2i-1}v_{2i}$ in this matching, there are $r^2$ possible ways the edges from $\{v_{2i-1},v_{2i}\}$ to $V^{**}_{2^r+1}$ can be coloured. 
Hence, for each $1 \le j \le 2^r$, there are subsets $V^{**}_j \subseteq V^*_j$ of size $\ell$ such that all edges between $V^{**}_j$ and $V^{**}_{2^r+1}$ have colour $c_{j, 2^r + 1}$, and such that there is a perfect matching in $H_{m,r}[V^{**}_{2i-1},V^{**}_{2i}]$ for each $1 \le i \le 2^{r-1}$. 

In summary, we have an induced subgraph $H^*_{m,r}$ of $H_{m,r}$ on $\cup _{j=1} ^{2^r+1} V^{**}_j$ with $|V_j ^{**}|\geq \ell$, a monochromatic perfect matching of colour $c_{2i-1,2i}$ in $H^*_{m,r}[V^{**}_{2i-1},V^{**}_{2i}]$ for all $1\leq i \leq 2^{r -1}$, and, for all other pairs $1 \le i<j\leq 2^{r}+1$, a monochromatic complete bipartite graph of colour $c_{i,j}$ between $V^{**}_i$ and $V^{**}_j$ in $H^*_{m,r}$.

Now consider an auxiliary copy of $K_{2^r+1}$ with vertex set $[2^{r}+1]$, colouring each edge $ij$ with colour $c_{i,j}$. By Fact~\ref{fact1} we find a monochromatic odd cycle $C_{\ell'}=x_1\dots x_{\ell '}x_1$ in $K_{2^r+1}$ where $2 \le \ell' \le 2^r + 1 \le \ell$.
It is easy to see that there is a homomorphism $\phi$ from $C_{\ell}$ to $C_{\ell'}$ so that for all but one of the edges $x_i x_{i+1}$, precisely one edge is mapped onto $x_i x_{i+1}$ by $\phi$.
We can thus ensure that all edges corresponding to parts with a perfect matching between them are mapped to exactly once, and it is then easy to lift the $C_{\ell'}$ to a monochromatic copy of $C_{\ell}$ in $H^*_{m,r}$, as desired.
\endproof

Returning to Theorem~\ref{thm:manycycles}, suppose that $d>1-2^{-r+1}$. Let $p=\omega (1/n^2)$, and let $m\geq m_0$ as in Lemma~\ref{l3}.
Let $G$ be a sufficiently large $n$-vertex graph of density at least $d$. Pairing up the parts of $H_{m,r}$ joined by a perfect matching gives a $(2^{r-1} + 1)$-partition of $V(H_{m,r})$ that shows $\rho_{2^{r-1} + 1}(H_{m,r}) \le 1/2$.  Theorem~\ref{thm:perturbedturan} then implies that w.h.p. $G_n \cup G(n,p)$ contains a copy of $H_{m,r}$. Thus, by Lemma~\ref{l3}, w.h.p. $G_n \cup G(n,p)$ is $(C_\ell,r)$-Ramsey. So indeed $p(n; r, C_\ell, d)= n^{-2}$ for all $1-2^{-r+1}< d \leq 1-2^{-r}$.

\smallskip
\subsubsection{The case when $1-2^{-r} <d\leq 1$.}

Suppose $d > 1 - 2^{-r}$, and $G$ is a sufficiently large graph of density $d$.  Since the graph $H_{m,r}$ from Lemma~\ref{l3} is $(2^{r}+1)$-partite, it follows from the Erd\H{o}s--Stone--Simonovits theorem~\cite{erdosstone} that $H_{m,r} \subseteq G$.  Hence, by Lemma~\ref{l3}, $G$ is $(C_{\ell},r)$-Ramsey, and thus $p(n; r, C_{\ell}, d) = 0$ for $d > 1 - 2^{-r}$, completing the proof of the theorem.

\subsection{Proof of Theorem~\ref{thm:cliquevscycle}}

In our final result, we establish the perturbed Ramsey thresholds for cliques versus odd cycles by showing $p(n;K_t, C_{\ell}, d) = n^{-2/(t-1)}$ for $t \ge 4$, odd $\ell \ge 5$ and $d \in (0,1/2]$.

\begin{proof}[Proof of Theorem~\ref{thm:cliquevscycle}]
First note that the lower bound follows from Observation~\ref{obs:trivial}(i).  Indeed, we can take $G$ to be the balanced complete bipartite $n$-vertex graph, which is $C_{\ell}$-free and of density at least $1/2$, and colour all its edges blue.  If $p =o( n^{-2/(t-1)})$, then with high probability $G(n,p)$ is $K_t$-free, and so we can colour the remaining edges of $G(n,p)$ red, obtaining an edge-colouring of $G \cup G(n,p)$ without any red $K_t$ or blue $C_{\ell}$.

\medskip

For the upper bound, let $p=\omega (n^{-2/(t-1)})$.  Apply Corollary~\ref{cor:regtuple} with $\delta := d, k := 2, r = 1$ and $\eps := (d/4)^{t + \ell}$ to the uncoloured graph $G$ of density $d$, obtaining an $\eps$-regular pair $(V_1, V_2)$ of density at least $d/2$, with $\card{V_1} = \card{V_2} = \eta n$ for some $\eta > 0$.  We shall show that $G \cup G(n,p)$ is $(K_t, C_{\ell})$-Ramsey provided $G(n,p)$ has the three properties given below, where $H$ is the vertex-disjoint union of $K_t$ and $C_{\ell}$, and $\alpha := (d/8)^{t + \ell}$:
\begin{itemize}
	\item[(a)] $G(n,p)[V_1]$ contains a copy of $H$ with at least $\alpha 2^{t + \ell} \card{V_2}$ common neighbours in $V_2$,
	\item[(b)] $G(n,p)$ is $\alpha \eta$-globally $(K_{t-1}, P_{\ell - 1})$-Ramsey, and
	\item[(c)] $G(n,p)$ is $\alpha \eta$-globally $(K_{t-2}, C_{\ell})$-Ramsey.
\end{itemize}

We first show that (a) holds with high probability.  By Lemma~\ref{lem:regpairs}(i), we know that at least half of all $(t + \ell)$-sets of vertices in $V_1$ have at least $(d/2 - \eps)^{t+\ell} \card{V_2}$ common neighbours in $V_2$, and $(d/2 - \eps)^{t +\ell} \ge \alpha 2^{t+\ell}$.  Let $\mc H$ be the collection of all possible copies of $H$ supported on such $(t + \ell)$-sets of $V_1$; it follows that $\card{\mc H} \ge \xi n^{t + \ell}$ for some constant $\xi = \xi(t,\ell,d) > 0$.  Furthermore, as $p = \omega (n^{-2/(t-1)})$, we have $\mu_1(H) = n^t p^{\binom{t}{2}} = \omega(1)$ (attained by taking $F = K_t$ in the definition of $\mu_1$).  Thus, by Theorem~\ref{thm:janson}, the probability that $G(n,p)[V_1]$ does not contain a suitable copy of $H$ is at most $\exp( - \xi \omega(1) / (2^{t + \ell + 1} (t + \ell)!)) = o(1)$.

For (b), we appeal to Theorem~\ref{thm:global1ramsey}.  Observe that $m_2(K_{t-1}) \ge m_2(P_{\ell - 1}) = 1$, $K_{t-1}$ is strictly balanced with respect to $m_2(\cdot, P_{\ell-1})$, and that $m_2(K_{t-1}, P_{\ell - 1}) = (t-1)/2$.  Thus there is some constant $C' = C'(\alpha \eta, K_{t-1}, P_{\ell - 1})$ such that for $p' \ge C' n^{-2/(t-1)}$, $G(n,p')$ satisfies (b) with high probability (so certainly our choice of $p$ satisfies (b) with high probability).

For (c), the case $t \ge 5$ can be handled with Theorem~\ref{thm:global1ramsey} just as above.  If $t = 4$, though, being $(K_{t-2}, C_{\ell})$-Ramsey is equivalent to containing $C_{\ell}$.  For this, we can apply Theorem~\ref{thm:janson} instead.  For every set $U \subset V(G(n,p))$ of size $\alpha \nu n$, let $\mc H_U$ be the possible copies of $C_{\ell}$ within $U$, observing that $\card{\mc H_U} = (\ell - 1)! \binom{\card{U}}{\ell} /2 = \xi' n^\ell$ for some appropriate constant $\xi' > 0$.  Moreover, as $p = \omega (n^{-2/3})$, we have $\mu_1(C_{\ell}) = n^{\ell} p^{\ell} = \omega ( n^{\ell / 3} ) \ge n^{5/3}$.  By Theorem~\ref{thm:janson}, the probability that $G(n,p)$ does not contain a copy of $C_{\ell}$ from $\mc H_U$ is at most $\exp( - \xi' n^{5 / 3} / (2^{\ell + 1} \ell!))$.  Taking a union bound over all possible choices of $U$, we see that with high probability, $G(n,p)$ will contain an $\ell$-cycle in each subset $U$, thus satisfying (c).

\medskip

To finish, let us see how these properties imply that $G \cup G(n,p)$ is $(K_t, C_{\ell})$-Ramsey.  Consider any $2$-colouring of $G \cup G(n,p)$.  By (a), we can find a copy $H_0$ of $H$ within $G(n,p)[V_1]$ with a set $W \subset V_2$ of at least $\alpha 2^{t + \ell} \card{V_2}$ common neighbours of $V(H_0)$.  For each $w \in W$, there are $2^{t + \ell}$ ways the edges between $w$ and $V(H_0)$ can be coloured, and so we can find a subset $U \subset W$ of $\alpha \card{V_2}$ 
common neighbours such that each vertex in $U$ has the same colour profile to $V(H_0)$ (i.e. the red neighbourhood in $V(H_0)$ of each $u \in U$ is the same; as is the blue neighbourhood).

First suppose there are $x, y \in V(H_0)$ such that the edges from $x$ to $U$ are all red while those from $y$ to $U$ are all blue.  By (b), we find  a red $K_{t-1}$ or a blue $P_{\ell - 1}$ in $G(n,p)[U]$.  Extending by either $x$ or $y$ respectively, we find either a red $K_t$ or a blue $C_{\ell}$ in $G \cup G(n,p)$, and are done.

Next, suppose all edges between $V(H_0)$ and $U$ are red.  Recall that $H_0$ contains $C_{\ell}$ as a subgraph.  If all edges of this $\ell$-cycle are blue, we are done, and hence there must be a red edge $\{x, y\}$.  By (c), $U$ must contain either a red $K_{t-2}$ or a blue $C_{\ell}$.  In the latter case, we are done, while in the former, extending the clique by $x$ and $y$ gives the desired red $K_t$.

Thus all edges between $V(H_0)$ and $U$ must be blue.  In this case, recall that $H_0$ contains $K_t$ as a subgraph.  If all its edges are red, then we are done, and hence there is a blue edge $\{u_0, u_1\}$.  Writing $\ell = 2k+1$, let $\{ u_2, u_3, \hdots, u_k \}$ be a set of $k-1$ other vertices from $V(H_0)$, and let $\{v_1, v_2, \hdots, v_k\}$ be an arbitrary set of $k$ vertices from $U$.  We then have a blue $\ell$-cycle $(u_0, u_1, v_1, u_2, v_2, \hdots, u_k, v_k, u_0)$ in $G \cup G(n,p)$, as required.
\end{proof}

\section{Concluding remarks and open problems}\label{conclude}

\begin{center}
``That's all there is, there isn't any more."
\end{center}
\begin{flushright}
--- Ethel Barrymore, \emph{Sunday} (1904)
\end{flushright}

In this paper we have determined how many random edges one must add to a dense graph to ensure w.h.p. the resulting graph is $(K_{t},K_{s})$-Ramsey for all $t \geq s\geq 5$; more precisely 
$p(n;K_t,K_s,d)=n^{-1/m_2 (K_t,K_{\lceil s/2 \rceil})}$ for all $0<d\leq 1/2$ (see Theorem~\ref{thm:bigcliques}). This complements work of Krivelevich,
Sudakov and Tetali~\cite{kst} who determined the corresponding threshold when $s=3$. 
This leaves only one open case for this range of densities; when  $s=4$. 
This is perhaps the most interesting open question resulting from our work. As mentioned previously, our work combined with Powierski's~\cite{power} settles this problem when $t = 4$.  The construction from~\cite{power} can be extended to the $(K_t, K_4)$-Ramsey question as well, and improves the lower bound from Proposition~\ref{prop:K4vsclique} for $t \in \{5,6\}$.  It would be of great interest to determine the correct threshold for $t \ge 5$.

Note Theorem~\ref{thm:bigcliques} does not consider all possible densities $d$;
indeed, it is still an open problem to determine how many random edges one must add to a graph $G$ of density $d$ to ensure it is w.h.p.  $(K_{t},K_{s})$-Ramsey if $d> 1-2/(s-1)$ and $s \leq t$. 
Notice though that if $d> 1-\frac{1}{r-1}$ where $r:=R(K_t,K_s)$, then by Tur\'an's theorem $G$ contains a copy of $K_r$ and thus is $(K_{t},K_{s})$-Ramsey.
Further, setting $r':= \lceil R(K_t,K_s)/2 \rceil$, Theorem~\ref{thm:perturbedturan} implies that if $d> 1- \frac{1}{r'-1}$ then  one only requires $p=\omega (n^{-2})$ to ensure $G \cup G(n,p)$  w.h.p. 
contains $K_r$ and thus is $(K_{t},K_{s})$-Ramsey.

We have completely  resolved the question of how many random edges one must add to a graph of fixed density to ensure w.h.p. the resulting graph is $(C_{k},C_{\ell})$-Ramsey for all $k, \ell \geq 3$ (see Theorem~\ref{thm:twocycles}).
Additionally we made some progress towards the analogous question for more colours via Theorem~\ref{thm:manycycles}; it would be interesting to close the one remaining gap in this statement (i.e. fully resolve the case when $1-2^{-r+1}< d \leq 1- 2^{-r+1}$).
Our work here is also related to a conjecture of Erd\H{o}s and Graham~\cite{eg} from 1973. Indeed,
recall that Theorem~\ref{thm:manycycles} only considers cycles $C_{\ell}$ that are sufficiently large compared to the number of colours $r$ (i.e. $\ell \geq 2^r+1$). 
This is because we apply Fact~\ref{fact1}: any $r$-colouring of $K_{2^{r}+1}$ yields a monochromatic odd cycle. In particular, this monochromatic cycle could have length  $2^r+1$.
Erd\H{o}s and Graham~\cite{eg} asked how large can the smallest monochromatic odd cycle in an $r$-colouring of $K_{2^{r}+1}$ be. Thus, getting non-trivial upper bounds on this question would yield a strengthening of 
Theorem~\ref{thm:manycycles}. We are not aware of such progress on this problem, though
 Day and Johnson~\cite{day} did prove the smallest such monochromatic cycle is unbounded  as $r$ grows,  thereby answering a question of Chung~\cite{chung}.

 Note that when the density $d$ of the graph $G$ is large, our upper bounds on $p(n;H_1,H_2,d)$ often come from the existence of a fixed-size subgraph that is $(H_1, H_2)$-Ramsey.  However, for small densities, the upper bounds use more global arguments, either using 
results on the Ramsey properties of the random graph,
 or finding large common neighbourhoods and then finding subgraphs there. It would be interesting to see to what extent this dichotomy extends to other cases of the Ramsey problem in randomly perturbed graphs.

Finally, there has been significant interest in Ramsey properties of random hypergraphs (see, for example,~\cite{conlongowers, frs, gnpsst}). It would be interesting to obtain analogues  of our results in the setting of randomly perturbed hypergraphs.

\section*{Acknowledgements}
We are grateful to Patrick Morris, whose advice was of great help in proving Claim~\ref{clm:viableforbid} from the appendix.
We are also grateful to the referee for a careful and helpful review.

\appendix

\section{Proofs of useful tools} \label{sec:app}

\begin{center}
``Imitation is the sincerest [form] of flattery."
\end{center}
\begin{flushright}
--- Charles Caleb Colton, \emph{Lacon, Or, Many Things in a Few Words: Addressed to Those who Think} (1824)
\end{flushright}

In this appendix we prove some of the results from Section~\ref{sec:tools} namely; Lemma~\ref{lem:multipartDRC}, a multipartite version of the dependent random choice lemma; Theorem~\ref{thm:janson}, which guarantees the existence of subgraphs within $G(n,p)$; Theorem~\ref{thm:global1ramsey}, which shows that when random graphs have certain Ramsey properties, they have them robustly and globally. These follow from either fairly standard arguments or minor modifications of existing proofs, but we include these results here for the sake of completeness.

\subsection{Proof of Lemma~\ref{lem:multipartDRC}} \label{sec:multipartDRCproof}

In Lemma~\ref{lem:multipartDRC}, copied below, we extend the basic lemma of dependent random choice from~\cite{fs}, showing that if a set of vertices $V_s$ is in $\eps$-regular pairs with other sets $V_1, V_2, \hdots, V_{s-1}$, then we can find a relatively large subset $U \subseteq V_s$ such that all small subsets of $U$ have many common neighbours in each of the other parts $V_i$.

\begingroup
\def\thetheorem{\ref{lem:multipartDRC}}
\begin{lemma}
Given $s \in \mathbb{N}$ and $\delta > 0$, suppose $0 < \eps \le \min\{ 1/(2s), \delta/2 \}$.  Let $V_1, V_2, \hdots, V_s$ be disjoint sets of vertices from a graph $G$, each of size $m$, such that for all $i \in [s-1]$, the pair $(V_i, V_s)$ is $\eps$-regular of density at least $\delta$.  For every $\beta > 0$ and $\ell \in \mathbb{N}$ there is some $\gamma = \gamma(s, \delta, \beta, \ell) > 0$ such that, if $m$ is sufficiently large, there is a subset $U_s \subseteq V_s$ of size at least $m^{1 - \beta}$ such that every set of $\ell$ vertices from $U_s$ has at least $\gamma m$ common neighbours in each of the sets $V_i$, $i \in [s-1]$.
\end{lemma}
\addtocounter{theorem}{-1}
\endgroup

\begin{proof}
For $i \in [s-1]$, let $T_i \subset V_i$ be the subset of vertices obtained when selecting $t$ vertices uniformly and independently (with repetition), where $t$ is to be determined.  Let $W_s := \{ v \in V_s : \cup_i T_i \subseteq N(v) \}$ be those vertices in $V_s$ adjacent to all chosen vertices.  For fixed $v \in V_s$, we have
\[ \mathbb{P}(v \in W_s) = \prod_{i=1}^{s-1} \left( \frac{\card{N(v) \cap V_i}}{m} \right)^t. \]
By Lemma~\ref{lem:regpairs}(i), with $\ell = 1$, the number of vertices in $V_s$ with at most $(\delta - \eps)m$ neighbours in $V_i$ for some $i \in [s-1]$ is at most $(s-1)\eps m$.  Hence for at least $(1 - (s-1)\eps)m$ vertices in $V_s$ we have $\card{N(v) \cap V_i} > (\delta - \eps)m$ for all $i \in [s-1]$.  Thus $\mathbb{E}[\card{W_s}] \ge (1 - (s-1)\eps)m (\delta - \eps)^{(s-1)t} \ge (m/2) (\delta/2)^{(s-1)t}$.

We now set $\gamma := (\delta / 2)^{2(s-1) \ell / \beta}$.  For every $\ell$-set $S \subseteq W_s$ with fewer than $\gamma m$ common neighbours in some $V_i$, remove one vertex of $S$ from $W_s$, and let $U_s$ be the set of remaining vertices.  Clearly every $\ell$-subset of $U_s$ has the requisite number of common neighbours, and so we need only show that $U_s$ is sufficiently large.

Observe that if $S \subseteq W_s$, then we must have $T_i \subseteq N(S) \cap V_i$ for each $i \in [s-1]$.  Thus, if $S$ has fewer than $\gamma m$ common neighbours in some part $V_i$, we can bound $\mathbb{P}(S \subseteq W_s) \le \gamma^t$.  In expectation, the number of vertices that needs to be removed from $W_s$ is thus at most $\binom{m}{\ell} \gamma^t$.  This leaves us with
\[ \mathbb{E}[\card{U_s}] \ge \frac{m}{2} \left( \frac{\delta}{2} \right)^{(s-1)t} - \binom{m}{\ell} \gamma^t \ge \frac{m}{2} \left( \frac{\delta}{2} \right)^{(s-1)t} - m^{\ell} \gamma^t. \]

Thus, if we set $t := \floor{\log (4 m^{-\beta}) / ( (s-1) \log ( \delta/2) )}$, the first term will be at least $2 m^{1- \beta}$, while our choice of $\gamma$ ensures $m^{\ell} \gamma^t = o(1)$.  Hence we can find some set $U_s \subset V_s$ of size at least $m^{1- \beta}$ with the desired property.
\end{proof}

\subsection{Proof of Theorem~\ref{thm:janson}} \label{sec:jansonproof}

The next result is a standard, straightforward, and very useful application of the Janson inequality~\cite{jlr}.

\begingroup
\def\thetheorem{\ref{thm:janson}}
\begin{thm}
Let $H$ be a graph with $v \ge 2$ vertices and $e \ge 1$ edges.  Let $[n]$ be the vertex set of $G(n,p)$, and, for some $\xi > 0$, let $\mc H$ be a collection of $\xi n^v$ possible copies of $H$ supported on $[n]$.  The probability that $G(n,p)$ does not contain any copy of $H$ from $\mc H$ is at most $\exp(- \xi \mu_1 / ( 2^{v+1} v!))$, where $\mu_1=\mu_1 (H) := \min \{ n^{v(F)} p^{e(F)} : F \subseteq H, e(F) \ge 1 \}$.
\end{thm}
\addtocounter{theorem}{-1}
\endgroup

We separate a calculation that will be needed in later proofs as well.  The quantity $\Delta$ bounds the expected number of pairs of distinct edge-intersecting copies of $H$, with the first copy from $\mc H$, that appear in $G(n,p)$, a quantity closely linked to the variance of the number of copies of $H$ from $\mc H$ in $G(n,p)$.

\begin{lemma} \label{lem:covariance}
Let $H$ be a graph with $v$ vertices and $e$ edges, let $\mc K_n(H)$ be the collection of all copies of $H$ in the complete graph $K_n$, and let $\mc H \subseteq \mc K_n(H)$ be a subcollection thereof.  We then have
\[ \Delta := \sum_{\substack{H' \in \mc H, H'' \in \mc K_n(H): \\ H' \neq H'', e(H' \cap H'') \neq 0}} \mathbb{P} \left( H' \cup H'' \subseteq G(n,p) \right) \le 2^v v! \card{\mc H} n^{v} p^{2e} / \mu_0, \]
where $\mu_0 = \mu_0(H) := \min \{ n^{v(F)} p^{e(F)} : F \subsetneq H, e(F) \ge 1 \}$.
\end{lemma}

\begin{proof}
We write $H' \sim H''$ if $H', H''$ are distinct copies of $H$ that share at least an edge, with $H' \in \mc H$ and $H'' \in \mc K_n(H)$.  Thus $\Delta$ denotes the expected number of (ordered) pairs $H' \sim H''$ with both $H'$ and $H''$ appearing in $G(n,p)$, which we want to bound from above.

Given $H' \in \mc H$, let us estimate its contribution to $\Delta$.  There are fewer than $2^v$ ways to choose a subset $S \subseteq V(H')$ of the vertices of $H'$ that are shared with $H''$, and at most $v!$ assignments $\varphi$ of these vertices to the vertices of $H''$.  This then determines the subgraph $F = F(S, \varphi) \subset H$ in which $H'$ and $H''$ intersect.  Let $\mc F := \{ (S, \varphi) : e(F(S,\varphi)) \ge 1 \}$ denote the set of viable pairs $(S, \varphi)$.  For each such pair, letting $F = F(S,\varphi)$, there are at most $n^{v - v(F)}$ choices for the remaining vertices of $H''$, and each such copy $H''$ introduces a further $e - e(F)$ edges.  Hence the probability that both $H'$ and $H''$ appear in $G(n,p)$ is $p^{2e - e(F)}$.  This gives
\[ \Delta \le \sum_{H' \in \mc H} \sum_{(S,\varphi) \in \mc F} n^{v - v(F)} p^{2e - e(F)} = n^v p^{2e} \sum_{H' \in \mc H} \sum_{(S, \varphi) \in \mc F} n^{- v(F)} p^{- e(F)}. \]

The summand is, by definition of $\mu_0$, at most $\mu_0^{-1}$.  Since there are at most $2^v v!$ choices of $(S, \varphi)$ in the inner sum, and $\card{\mc H}$ choices of $H'$ in the outer sum, we get the desired bound of 
\[ \Delta \le 2^v v! \card{\mc H} n^v p^{2e} / \mu_0. \qedhere \]
\end{proof}

Theorem~\ref{thm:janson} follows almost immediately from the lemma.

\begin{proof}[Proof of Theorem~\ref{thm:janson}]
Let $X_{\mc H}$ be the random variable counting the number of copies of $H$ from $\mc H$ that appear in $G(n,p)$, so that we seek to bound $\mathbb{P}(X_{\mc H} = 0)$ from above.  We clearly have $\mu := \mathbb{E}[X_{\mc H}] = \card{\mc H} p^e = \xi n^v p^e$.  Thus, letting $\Delta$ and $\mu_0$ be as in Lemma~\ref{lem:covariance}, we have $\Delta \le 2^v v! \mu^2 / (\xi \mu_0)$.  Furthermore, we have $\mu_1 = \min \{ \mu / \xi, \mu_0 \}$.

By the Janson inequality (see~\cite[Theorem 8.1.1]{as}), if $\Delta \le \mu$, then
\[ \mathbb{P}(X_{\mc H} = 0) \le \exp( - \mu / 2) \le \exp( - \xi \mu_1 / (2^{v+1} v!) ). \]
Otherwise, when $\Delta > \mu$, we can instead apply the extended Janson inequality~\cite[Theorem 8.1.2]{as}:
\[ \mathbb{P}(X_{\mc H} = 0) \le \exp( - \mu^2 / (2 \Delta)) \le \exp( - \xi \mu_0 / (2^{v+1} v!)) \le \exp ( - \xi \mu_1 / (2^{v+1} v!) ). \qedhere \]
\end{proof}

\subsection{Proof of Theorem~\ref{thm:global1ramsey}} \label{sec:global1ramseyproof}

We finally justify our strengthening of the $1$-statement of the asymmetric random Ramsey problem, showing that above the threshold, $G(n,p)$ is both robustly and globally $(H_1,H_2)$-Ramsey.  We recall the precise statement below.

\begingroup
\def\thetheorem{\ref{thm:global1ramsey}}
\begin{thm}[\cite{gnpsst,hst}]
Let $H_1$ and $H_2$ be graphs such that $m_2(H_1) \ge m_2(H_2) \ge 1$, and 
\begin{itemize}
	\item[(a)] $m_2(H_1) = m_2(H_2)$, or 
	\item[(b)] $H_1$ is strictly balanced with respect to $m_2(\cdot, H_2)$.
\end{itemize}

The random graph $G(n,p)$ then has the following Ramsey properties:
\begin{itemize}
	\item[(i)] There are constants $\gamma = \gamma(H_1, H_2) > 0$ and $C_1 = C_1(H_1, H_2)$ such that if $p \ge C_1 n^{-1/m_2(H_1, H_2)}$ and, for $i \in [2]$, $\mc F_i \subseteq \binom{[n]}{v(H_i)}$ is a collection of at most $\gamma n^{v(H_i)}$ forbidden subsets, then $G(n,p)$ is with high probability robustly $(H_1, H_2)$-Ramsey with respect to $(\mc F_1, \mc F_2)$.
	\item[(ii)] For every $\mu > 0$ there is a constant $C_2 = C_2(H_1, H_2, \mu)$ such that if $p \ge C_2 n^{-1/m_2(H_1, H_2)}$, then $G(n,p)$ is with high probability $\mu$-globally $(H_1, H_2)$-Ramsey.
	\item[(iii)] If we further have $m_2(H_2) > 1$, then there are constants $\beta_0 = \beta_0(H_1, H_2) > 0$ and $C_3 = C_3(H_1, H_2)$ such that if $0 \le \beta \le \beta_0$ and $p \ge C_3 n^{-(1-\beta)/m_2(H_1, H_2)}$, then with high probability $G(n,p)$ is $n^{-\beta}$-globally $(H_1, H_2)$-Ramsey.
\end{itemize}
\end{thm}
\addtocounter{theorem}{-1}
\endgroup

\subsubsection{Containers for nearly-$H$-free graphs and Ramsey supersaturation}

As stated earlier, we shall prove this result by retracing the proofs of Hancock, Staden and Treglown~\cite{hst} and of Gugelmann, Nenadov, Person, \v{S}kori\'c, Steger and Thomas~\cite{gnpsst}, making minor modifications along the way.  One commonality between the two proofs is the use of the hypergraph containers theorems of Balogh, Morris and Samotij~\cite{bms} and of Saxton and Thomason~\cite{st}.

Indeed, if $G(n,p)$ is not $(H_1, H_2)$-Ramsey, then its edges can be partition into a red $H_1$-free subgraph and a blue $H_2$-free subgraph.  The aforementioned containers theorems show there are only a small number of ``containers" that these $H_i$-free subgraphs must belong to, and one may use the properties of these containers to show it is very unlikely that $G(n,p)$ admits such a partition.

However, in the robust Ramsey setting, it is no longer true that the red subgraph must be $H_1$-free or that the blue subgraph must be $H_2$-free, as they may contain some forbidden copies of these subgraphs.  Fortunately, the containers do not only capture $H_i$-free graphs, but also graphs with relatively few copies of $H_i$.  To make this precise, we use the formulation of Saxton and Thomason~\cite{st}, for which we require the following definition concerning $r$-uniform hypergraphs (or $r$-graphs).

\begin{definition}[Co-degree function]
Let $\Gamma$ be a vertex-transitive $d$-regular $r$-graph, and let $\tau > 0$.  Given a set $S$ of at most $r$ vertices, define $d(S)$ to be the number of edges of $\Gamma$ containing $S$.  For each $2 \le j \le r$, we define $\delta_j$ by the equation
\[ \delta_j \tau^{j-1} d = \max_{S \in \binom{V(\Gamma)}{j}} d(S). \]
The \emph{co-degree function} $\delta(\Gamma, \tau)$ is then defined by
\[ \delta(\Gamma, \tau) := 2^{\binom{r}{2} - 1} \sum_{j=2}^r 2^{-\binom{j-1}{2}} \delta_j. \]
\end{definition}
Note that the notion of co-degree can be defined more generally (i.e. not just for vertex-transitive $r$-graphs); we only state the definition in this form to simplify things a little.
In the case of vertex-transitive graphs, Corollary 3.6 in~\cite{st} reads as below.

\begin{corollary}[Saxton--Thomason~\cite{st}] \label{cor:containers}
Let $\Gamma$ be a  vertex-transitive $r$-graph on vertex set $[N]$.  Let $0 < \eps, \tau < 1/2$.  Suppose that $\tau$ satisfies $\delta(\Gamma,\tau) \le \eps / 12r!$.  Then there exists a constant $c = c(r)$ and a function $\psi: \mc P([N])^q \rightarrow \mc P([N])$, where $q \le c \log (1 / \eps)$, with the following properties.  Let $\mc T: = \{ (T^{(1)}, \hdots, T^{(q)}) \in \mc P([N])^q : \card{T^{(j)}} \le c \tau N, j \in [q] \}$, and let $\mc C := \{ \psi(T) : T \in \mc T \}$.  Then
\begin{itemize}
	\item[(a)] for every set $I \subset [N]$ for which $e(\Gamma[I]) \le 24 \eps r! r \tau^r e(\Gamma)$, there exists $T = (T^{(1)}, \hdots, T^{(q)}) \in \mc T \cap \mc P(I)^q$ with $I \subset \psi(T) \in \mc C$,
	\item[(b)] $e(\Gamma[\chi]) \le \eps e(\Gamma)$ for all $\chi \in \mc C$, and
	\item[(c)] $\log \card{\mc C} \le c \log (1 / \eps) N \tau \log(1/\tau)$.
\end{itemize}
\end{corollary}

Given a graph $H$, we wish to apply Corollary~\ref{cor:containers} to build containers for $n$-vertex graphs that are nearly $H$-free.  To that end, we set $N := \binom{n}{2}$ and $r := e(H)$, and take $\Gamma$ to be the $r$-graph whose vertices correspond to the edges of $K_n$ and whose hyperedges are the edge-sets of all copies of $H$ in $K_n$.  It is clear that $\Gamma$ is vertex-transitive, and, writing $v := v(H)$ and $e := e(H)$, each vertex in $\Gamma$ has degree
\[ d := \binom{n}{v} \frac{v!}{\card{\Aut(H)}} \frac{e}{\binom{n}{2}} = \Theta \left( n^{v-2} \right). \]

We next must bound $d(S)$ for sets $S \in \binom{V(\Gamma)}{j}$, where $2 \le j \le r$.  Clearly, if the set $S$ of edges spans $a$ vertices of $H$, then $d(S) = \Theta \left(n^{v-a} \right)$.  Consequently,
\[ \max_{S \in \binom{V(\Gamma)}{j}} d(S) = \Theta \left( n^{v- a_j} \right), \]
where $a_j$ is the minimum number of vertices that a set of $j$ edges of $H$ can span.  Hence
\[ \delta_j := \max_{S \in \binom{V(\Gamma)}{j}} d(S) / \left( \tau^{j-1} d \right) = \Theta \left( \tau^{1 - j} n^{2 - a_j} \right), \]
and so $\delta(\Gamma, \tau) = \Theta \left( \max_j \tau^{1-j} n^{2 - a_j} \right)$.  In order to apply Corollary~\ref{cor:containers}, we need the co-degree function to satisfy $\delta(\Gamma,\tau) \le \eps / 12r!$, where $\eps$ will be a small constant, and therefore we cannot let $\tau$ be too small.  Indeed, we need $\tau = \Omega \left( n^{- (a_j - 2) / (j-1) } \right)$ for all $j$, which is equivalent to saying that for every subgraph $F \subseteq H$, we need $\tau = \Omega \left( n^{ - (v(F) - 2)/(e(F) - 1)} \right)$.  Recalling the definition of the $2$-density $m_2(H)$, this implies that for every $0 < \eps < 1/2$, we may apply Corollary~\ref{cor:containers} with $\tau = O_{\eps} \left( n^{-1/m_2(H)} \right)$.  Substituting this choice of $\tau$ into Corollary~\ref{cor:containers} gives the following.

\begin{corollary} \label{cor:Hcontainers}
Given a graph $H$ and a constant $0 < \eps < 1/2$, there are positive constants $q = q(\eps, H), \beta = \beta(\eps, H)$ and $\kappa = \kappa(\eps, H)$ such that for every $n \in \mathbb{N}$, there is a function $\psi: \mc P(E(K_n))^q \rightarrow \mc P(E(K_n))$ with the following properties.  Let $\mc T := \{ (T^{(1)}, \hdots, T^{(q)}) \in \mc P(E(K_n))^q : \card{T^{(j)}} \le \beta n^{2 - 1/m_2(H)}, j \in [q] \}$, and let $\mc C := \{ \psi(T) : T \in \mc T \}$ be the family of corresponding subgraphs of $K_n$.  Then
\begin{itemize}
	\item[(i)] for every subgraph $F \subseteq K_n$ with at most $\kappa n^{v(H) - e(H)/m_2(H)}$ copies of $H$, there is some $T \in \mc T \cap \mc P(E(F))^q$ with $F \subseteq \psi(T) \in \mc C$, and
	\item[(ii)] each subgraph $\psi \in \mc C$ contains at most $\eps \binom{n}{v(H)}$ copies of $H$.
\end{itemize}
\end{corollary}

Another result we shall need is a supersaturated version of Ramsey's Theorem --- when colouring the edges of a large enough clique, we do not just see a monochromatic copy of a given graph, but we see many such copies.  This follows from the classic Ramsey's Theorem by a standard averaging argument, the  details of which can be found in~\cite{gnpsst}.

\begin{prop}[Folklore] \label{prop:supersatramsey}
Given $r \ge 2$ and graphs $F_1, F_2, \hdots, F_r$, there are positive constants $n_0 = n_0(F_1, \hdots, F_r)$ and $\pi = \pi(F_1, \hdots, F_r)$ such that if $n \ge n_0$, then every $r$-colouring of $K_n$ has some $i \in [r]$ for which there are at least $\pi \binom{n}{v(F_i)}$ monochromatic copies of $F_i$ in colour $i$.
\end{prop}

\subsubsection{Random graphs are robustly Ramsey: Case (a)}

We now prove the robust Ramsey result under the assumption (a); that is, $m_2(H_1) = m_2(H_2) \ge 1$.  Recall that we then have $m_2(H_1, H_2) = m_2(H_1)$ as well.  We wish to show that there are constants $C, \gamma > 0$ such that if $p \ge C n^{-1/m_2(H_1)}$ and, for $i \in [2]$, $\mc F_i$ is a collection of at most $\gamma n^{v(H_i)}$ forbidden copies of $H_i$, then $G(n,p)$ is with high probability robustly $(H_1, H_2)$-Ramsey with respect to $(\mc F_1, \mc F_2)$.  By monotonicity, we may assume $p = C n^{-1/m_2(H_1)}$ and, for each $i$, $\card{\mc F_i} = \gamma n^{v(H_i)}$.

Recall that the only copies of $H_i$ that may be monochromatic in their respective colours are the forbidden ones from $\mc F_i$.  We first claim that there are few forbidden copies of each $H_i$ in $G(n,p)$.

\begin{claim} \label{clm:fewforbidden}
If $p = C n^{-1/m_2(H_i)}$ for some constant $C\geq 1$, then for each $i \in [2]$, with high probability $G(n,p)$ contains at most $2 \gamma C^{e(H_i)} n^{v(H_i) - e(H_i)/m_2(H_i)}$ copies of $H_i$ from $\mc F_i$.
\end{claim}

\begin{proof}
We begin with a preliminary calculation.  Since $p = C n^{-1/m_2(H_i)} \ge n^{-1}$, for any subgraph $F \subseteq H_i$ with at least one edge we have
\begin{equation} \label{eqn:subgraphcount}
n^{v(F)} p^{e(F)} = n^2 p \cdot n^{v(F) - 2} p^{e(F) - 1} \ge C^{e(F) - 1} n^2 p = \Omega(n).
\end{equation}

Now let $X_i$ denote the number of forbidden copies of $H_i$ from $\mc F_i$ that appear in $G(n,p)$.  We clearly have $\mu(H_i) := \mathbb{E}[X_i] = \card{\mc F_i} p^{e(H_i)} = \gamma n^{v(H_i)} p^{e(H_i)} = \gamma C^{e(H_i)} n^{v(H_i) - e(H_i)/m_2(H_i)}$.  Moreover, taking $F = H_i$ in~\eqref{eqn:subgraphcount}, we have $\mu(H_i) = \Omega(n)$.

We next wish to estimate the variance of $X_i$.  Since edge-disjoint copies of $H_i$ appear independently of each other, we can bound 
\[ \Var(X_i) \le \sum_{\substack{F_1, F_2 \in \mc F_i: \\ e(F_1 \cap F_2) \ge 1}} \mathbb{P}\left(F_1 \cup F_2 \subseteq G(n,p) \right). \]

The diagonal terms, when $F_1 = F_2$, contribute a total of $\mu(H_i)$, which, since $\mu(H_i) = \Omega(n)$, is $O( \mu(H_i)^2 / n )$.  On the other hand, the sum over ordered pairs $(F_1, F_2)$ with $F_1 \neq F_2$ can be bounded by $\Delta$ from Lemma~\ref{lem:covariance}, applied with $H := H_i$ and $\mc H := \mc F_i$.  The lemma gives $\Delta \le 2^{v(H_i)} v(H_i)! \card{\mc F_i} n^{v(H_i)} p^{2e(H_i)} / \mu_0(H_i) = \Theta( \mu(H_i)^2 / \mu_0(H_i) )$.  By~\eqref{eqn:subgraphcount} we have $\mu_0(H_i) = \Omega(n)$, from which we can deduce $\Delta = O(\mu(H_i)^2 / n)$.

Thus $\Var(X_i) = O( \mu(H_i)^2 / n )$.  Chebyshev's inequality now gives the desired result, since 
\[ \mathbb{P} \left( X_i \ge 2 \gamma C^{e(H_i)} n^{v(H_i) - e(H_i) / m_2(H_i)} \right) = \mathbb{P} \left( X_i \ge 2 \mu(H_i) \right) \le \frac{\Var(X_i) }{ \mu(H_i)^2 } = O\left(\frac{1}{n}\right). \qedhere \]
\end{proof}

Define $\mc E_0$ to be the event that $G(n,p)$ contains more than $2 \gamma C^{e(H_i)} n^{v(H_i) - e(H_i) / m_2(H_i)}$ copies of $H_i$ from $\mc F_i$, for some $i \in [2]$.  Claim~\ref{clm:fewforbidden} shows that $\mc E_0$ almost surely does not occur, and we shall henceforth assume it does not.

Now suppose $G(n,p)$ is not robustly $(H_1, H_2)$-Ramsey with respect to $(\mc F_1, \mc F_2)$.  This means the edges of $G(n,p)$ may be $2$-coloured so that the red copies of $H_1$ and the blue copies of $H_2$ all belong to $\mc F_1$ and $\mc F_2$ respectively.  In light of Claim~\ref{clm:fewforbidden}, the red subgraph has few copies of $H_1$, and hence by Corollary~\ref{cor:Hcontainers} must lie within one of the containers for nearly-$H_1$-free graphs.  Similarly, the blue subgraph must lie within one of the containers for nearly-$H_2$-free graphs.  Since there are very few containers, and the containers are relatively small, this is very unlikely.

To be more precise, let $\pi = \pi(H_1, H_2, K_2)$ be the constant from Proposition~\ref{prop:supersatramsey}.  For $i \in [2]$, we then apply Corollary~\ref{cor:Hcontainers} with $\eps := \pi/2$ and $H_i$ to get constants $q_i$, $\beta_i$ and $\kappa_i$, and a map $\psi_i$ from a collection of sequences $\mc T_i := \{ (T_i^{(1)}, \hdots, T_{i}^{(q_i)}) \in \mc P(E(K_n))^{q_i} : \card{T_i^{(j)}} \le \beta_i n^{2 - 1/m_2(H_i)}, j \in [q_i] \}$ to the set of containers $\mc C_i := \{ \psi_i(T_i) : T_i \in \mc T_i \}$.  Given a choice of the constant $C$, which we shall implicitly specify at the end, we set $\gamma := \min \{ \kappa_i / (2 C^{e(H_i)}) : i \in [2] \}$.

\medskip

Now fix a $2$-colouring of $G(n,p)$ where all the red copies of $H_1$ and blue copies of $H_2$ belong to $\mc F_1$ and
$\mc F_2$ respectively.  By Claim~\ref{clm:fewforbidden}, the red subgraph $G_1$ has at most $2 \gamma C^{e(H_1)} n^{v(H_1) - e(H_1)/m_2(H_1)}$ copies of $H_1$, which, by our choice of $\gamma$, is at most $\kappa_1 n^{v(H_1) - e(H_1)/m_2(H_1)}$.  By Corollary~\ref{cor:Hcontainers}, there is some $T_1 \in \mc T_1 \cap \mc P(E(G_1))^{q_1}$ such that $G_1 \subseteq C_1 := \psi_1(T_1) \in \mc C_1$.  A similar argument for the blue subgraph $G_2$ gives $T_2 \in \mc T_2 \cap \mc P(E(G_2))^{q_2}$ with $G_2 \subseteq C_2 := \psi_2(T_2) \in \mc C_2$.  Since $G(n,p) = G_1 \cup G_2$, we have $G(n,p) \subseteq C_1 \cup C_2$.

Corollary~\ref{cor:Hcontainers} also asserts that, for $i \in [2]$, $C_i$ contains at most $\pi \binom{n}{v(H_i)} / 2$ copies of $H_i$.  By Proposition~\ref{prop:supersatramsey}, it follows that there are at least $\pi \binom{n}{2}$ edges in $R(T_1, T_2) := E(K_n) \setminus (C_1 \cup C_2)$, all of which must be missing from $G(n,p)$.

\medskip

Given $T_1 \in \mc T_1$ and $T_2 \in \mc T_2$, we define the shorthand $T_1 \cup T_2 := \cup_{i \in [2]} \left\{ \cup_{j \in [q_i]} T_i^{(j)} \right\}$.  Now let $\mc E(T_1, T_2)$ be the event that all edges in $T_1 \cup T_2$ are in $G(n,p)$ while all edges in $R(T_1, T_2)$ are missing.  Since $\card{R(T_1, T_2)} \ge \pi \binom{n}{2}$, we have
\[ \mathbb{P}(\mc E(T_1, T_2)) = p^{\card{T_1 \cup T_2}} (1-p)^{\card{R(T_1, T_2)}} \le p^{\card{T_1 \cup T_2}} \exp \left( - \pi \binom{n}{2} p \right). \]

 Our above discussion shows that if $G(n,p)$ does not have the desired robust Ramsey property, then either $\mc E_0$ or one of the events $\mc E(T_1, T_2)$ must occur.  By the union bound, the probability of $G(n,p)$ not being robustly $(H_1, H_2)$-Ramsey with respect to $(\mc F_1, \mc F_2)$ is at most
\begin{align*}
	\sum_{T_1 \in \mc T_1, T_2 \in \mc T_2} \mc E(T_1, T_2) + \mathbb{P}(\mc E_0) &\le \exp \left( - \pi \binom{n}{2} p \right) \sum_{T_1 \in \mc T_1, T_2 \in \mc T_2} p^{\card{T_1 \cup T_2}} + o(1) \\
	&= \exp \left( - \pi \binom{n}{2} p \right) \sum_{S \subseteq E(K_n)} \sum_{\substack{T_1 \in \mc T_1, T_2 \in \mc T_2 : \\ T_1 \cup T_2 = S }} p^{\card{S}} + o(1).
\end{align*}

Now recall that $T_i = (T_i^{(1)}, \hdots, T_i^{(q_i)})$, where $\card{T_i^{(j)}} \le \beta_i n^{2-1/m_2(H_i)}$, and so $\card{T_1 \cup T_2} \le M_n := (\beta_1 q_1 + \beta_2 q_2) n^{2 - 1/m_2(H_1)}$.  Moreover, given a set $S$ of at most this size, for each edge $e \in S$, we must choose which of the sets $T_i^{(j)}$ it belongs to, and hence there are at most $2^{(q_1 + q_2)\card{S}}$ pairs $(T_1, T_2)$ corresponding to $S$.  Therefore
\[ \sum_{S \subseteq E(K_n)} \sum_{\substack{T_1 \in \mc T_1, T_2 \in \mc T_2: \\ T_1 \cup T_2 = S}} p^{\card{S}} \le \sum_{\substack{S \subseteq E(K_n) : \\ \card{S} \le M_n}} \left( 2^{q_1 + q_2} p \right)^{\card{S}} \le \sum_{s = 0}^{M_n} \binom{n^2}{s} \left( 2^{q_1 + q_2} p \right)^s \le \sum_{s=0}^{M_n} \left( \frac{ 2^{q_1 + q_2} e n^2 p }{s} \right)^s. \]

Substituting $p = Cn^{-1/m_2(H_1)}$, we find that the summand is $(2^{q_1 + q_2} C e n^{2 - 1/m_2(H_1)} / s)^s$.  Since the function $f(x) = (A/x)^x$ increases to a maximum at $x = A/e$, it follows that, if $C$ is sufficiently large, the largest summand corresponds to $s = M_n$.  We thus have
\[ \sum_{T_1 \in \mc T_1, T_2 \in \mc T_2} p^{\card{T_1 \cup T_2}} \le (M_n+1) \left( \frac{2^{q_1 + q_2} C e n^{2 - 1/m_2(H_1)}}{M_n} \right)^{M_n} = (M_n+1) \exp(L_C n^{2 - 1/m_2(H_1)}), \]
where $L_C := (\beta_1 q_1 + \beta_2 q_2) \ln \left(2^{q_1 + q_2} C e / (\beta_1 q_1 + \beta_2 q_2) \right)$.  This gives
\begin{align} \label{eqn:errorbounda}
	\mathbb{P} \left( G(n,p) \textrm{ is not robustly Ramsey } \right) &\le (M_n + 1) \exp\left( L_C n^{2 - 1/m_2(H_1)} - \pi \binom{n}{2} p \right) + o(1) \notag \\ 
	&\le (M_n + 1) \exp\left( (L_C - C \pi / 4) n^{2 - 1/m_2(H_1)} \right) + o(1) \notag \\
	&\le \exp\left( - C \pi n^{2 - 1/m_2(H_1)} / 8 \right) + o(1),
\end{align}
where the final inequality holds when we choose $C$ to be sufficiently large, since $L_C$ is logarithmic in $C$, and if $n$ is large enough, since $M_n$ is only polynomial in $n$.  This shows $G(n,p)$ is with high probability robustly $(H_1, H_2)$-Ramsey with respect to $(\mc F_1, \mc F_2)$.
\qed

\subsubsection{Random graphs are robustly Ramsey: Case (b)}

We next prove the robust Ramsey result under the assumption (b), when $H_1$ is strictly balanced with respect to $m_2(\cdot, H_2)$.  When $m_2(H_1) > m_2(H_2)$ and $p = C n^{-1/m_2(H_1, H_2)}$, there are too many nearly-$H_1$-free containers for the union bound calculation of Case (a) to work.

We will instead exploit the fact that nearly all copies of $H_1$ in $G(n,p)$ are edge-disjoint.  To that end, call a copy of $H_1$ in $G(n,p)$ \emph{isolated} if it is edge-disjoint from all other copies of $H_1$, and call it \emph{non-isolated} otherwise.  Our first claim shows that there are few non-isolated copies of $H_1$.

\begin{claim} \label{clm:fewnoniso}
Let $H_1$ be strictly balanced with respect to $m_2(\cdot, H_2)$, and let $p = C n^{-1/m_2(H_1, H_2)}$ for some constant $C\geq 1$.  There is a constant $\delta = \delta(H_1, H_2)$ such that the number of non-isolated copies of $H_1$ in $G(n,p)$ is with high probability at most $n^{2 - 1/m_2(H_2) - \delta}$.
\end{claim}

\begin{proof}
If a copy $H'$ of $H_1$ is non-isolated, then there is some other copy $H'' \neq H'$ of $H_1$ such that $e(H' \cap H'') \ge 1$.  The number of non-isolated copies of $H_1$ in $G(n,p)$ is thus bounded from above by the expected number of such pairs $\{ H', H'' \}$, which is precisely the quantity $\Delta$ from Lemma~\ref{lem:covariance}, when $H:= H_1$ and $\mc H := \mc K_n(H_1)$.  The lemma gives 
\[ \Delta \le \frac{ 2^{v(H_1)} v(H_1)! \card{\mc K_n(H_1)} n^{v(H_1)} p^{2e(H_1)} }{ \mu_0(H_1) } \le \frac{ 2^{v(H_1)} v(H_1)! (n^{v(H_1)} p^{e(H_1)})^2 }{ \mu_0(H_1) },\]
where $\mu_0(H_1) = \min \{ n^{v(F)} p^{e(F)} : F \subsetneq H_1, e(F) \ge 1\}$.

For each $F \subseteq H_1$ with $e(F) \ge1$, define $\delta_{F}$ such that
\[ m_2(H_1, H_2) = \frac{e(H_1)}{v(H_1) - 2 + 1/m_2(H_2)} =: \frac{e(F)}{v(F) - 2 + 1/m_2(H_2) - 2 \delta_{F}}. \]
It then follows that for every $F \subseteq H_1$ we have
\begin{equation} \label{eqn:mu0calc}
n^{v(F)} p^{e(F)} = C^{e(F)} n^{v(F) - e(F) (v(H_1) - 2 + 1/m_2(H_2)) / e(H_1)} = C^{e(F)} n^{2 - 1/m_2(H_2) + 2 \delta_{F}}.
\end{equation}

We clearly have $\delta_{H_1} = 0$.  On the other hand, since $H_1$ is strictly balanced with respect to $m_2( \cdot, H_2)$, we have $\delta_{F} > 0$ in all other cases.  Letting $\delta = \delta(H_1, H_2) := \min \{ \delta_{F} : F \subsetneq H_1, e(F) \ge 1 \}$, it follows that $\mu_0(H_1) \ge n^{2 - 1/m_2(H_2) + 2\delta}$.

By using these values for $\delta_{H_1}$ and $\mu_0(H_1)$, we can bound the expected number of pairs of intersecting copies of $H_1$ by
\[ \Delta = O \left( \frac{(n^{v(H_1)} p^{e(H_1)})^2 }{\mu_0(H_1)} \right) = O \left( \frac{(n^{2 - 1/m_2(H_2)} )^2}{n^{2 - 1/m_2(H_2) + 2\delta}} \right) = O \left( n^{2 - 1/m_2(H_2) - 2\delta} \right). \]

Using Markov's inequality, the probability that the number of pairs of intersecting copies of $H_1$, and thus the number of non-isolated copies of $H_1$, in $G(n,p)$ is more than $n^{2 - 1/m_2(H_2) - \delta}$ is $O(n^{-\delta})$, establishing the claim.
\end{proof}

Let $\mc E_1$ be the event that $G(n,p)$ has more than $n^{2 - 1/m_2(H_2) - \delta}$ non-isolated copies of $H_1$, which Claim~\ref{clm:fewnoniso} shows we may assume does not occur.  We denote by $L$ the set of edges in non-isolated copies of $H_1$.  Note that any edge $e \in E(G(n,p)) \setminus L$ is contained in at most one copy of $H_1$.

\medskip

Now suppose that $G(n,p)$ is not robustly $(H_1, H_2)$-Ramsey with respect to $(\mc F_1, \mc F_2)$, and fix a $2$-colouring of its edges where the only red copies of $H_1$ and blue copies of $H_2$ in $G(n,p)$ are  forbidden (i.e. in $\mc F_1$ or $\mc F_2$ respectively).  We may safely assume that any edge of $G(n,p)$ that is not contained in a copy of $H_1$ is red, since if not, recolouring it red does not create a red copy of $H_1$.  Furthermore, we may assume that every isolated non-forbidden copy of $H_1$ has exactly one blue edge --- since it is not in $\mc F_1$, it cannot be monochromatic red, and if it has more than one blue edge, then since it is an isolated copy of $H_1$, recolouring all but one of them red cannot create a red copy of $H_1$.

Let $B$ be the subgraph of $G(n,p)$ consisting of all edges not in $L$ that are coloured blue.  Note that each edge of $B$ must be contained in a copy of $H_1$, and moreover these copies are pairwise edge-disjoint.  We thus call a set $E$ of edges in $G(n,p)$ \emph{viable} if each $e \in E$ is contained in some copy $W_e$ of $H_1$ such that the copies $\{ W_e : e \in E \}$ are pairwise edge-disjoint.  We further call such a set of copies a \emph{witness} for the viability of $E$.

Note that a set of edges $E$ can only be contained in $B$ if $E$ is viable.  The next claim bounds the likelihood of viability.

\begin{claim} \label{clm:viableprob}
If $p = C n^{- 1/m_2(H_1, H_2)}$ for some positive constant $C$, then a set $E$ of $s$ edges is viable with probability at most $\rho^s$, where $\rho := C^{e(H_1)} v(H_1)^2 n^{-1/m_2(H_2)}$.
\end{claim}

\begin{proof}
We take a union bound over all possible witnesses for $E$.  Since the copies $W_e$ of $H_1$ in a witness are edge-disjoint, each witness requires exactly $e(H_1) s $ edges to appear in $G(n,p)$, and hence appears with probability $p^{e(H_1) s}$.

To estimate the number of possible witnesses, note that for each edge $e \in E$, we have fewer than $v(H_1)^2$ ways to assign vertices of $H_1$ to the vertices of $e$ in the copy $W_e$.  There are then at most a further $n^{v(H_1) - 2}$ choices for the remaining vertices in $W_e$.  This shows there are at most $(v(H_1)^2 n^{v(H_1) - 2})^s$ possible witnesses for $E$.

The expected number of witnesses for $E$, which bounds from above the probability of $E$ being viable, is therefore at most
\[ \left( v(H_1)^2 n^{v(H_1) - 2} \right)^s p^{e(H_1) s} = \left( v(H_1)^2 n^{v(H_1) - 2} p^{e(H_1)} \right)^s = \rho^s, \]
where in the last equality we use the fact that $p^{e(H_1)} = C^{e(H_1)} n^{2 - v(H_1) - 1/m_2(H_2)}$.
\end{proof}

Only forbidden copies of $H_2$ can be monochromatically blue.  Furthermore, a copy of $H_2$ in $B$ must also be viable.  Claim~\ref{clm:viableforbid} asserts there cannot be many viable forbidden copies of $H_2$.  We will defer its proof for the moment, first seeing how it implies the robust Ramsey result.

\begin{claim} \label{clm:viableforbid}
The number of viable forbidden copies of $H_2$ from $\mc F_2$ is with high probability at most $\gamma K n^{v(H_2) - e(H_2)/m_2(H_2)}$ for some constant $K = K(C, H_1, H_2)$.
\end{claim}

We now lay some groundwork for the rest of the proof, including specifying our choice of the constant $\gamma$.  Let $\pi = \pi(H_1, H_2, K_2)$ be the constant from Proposition~\ref{prop:supersatramsey}.  We now apply Corollary~\ref{cor:Hcontainers} with $H := H_2$ and $\eps := \pi /2$.  This gives us constants $q, \beta$ and $\kappa$, together with a collection $\mc T := \{(T^{(1)}, \hdots, T^{(q)}) \in \mc P(E(K_n))^q : \card{T^{(j)}} \le \beta n^{2 - 1/m_2(H_2)}, j \in [q] \}$ and a set of containers $\mc C := \{ \psi(T) : T \in \mc T\}$ for nearly-$H_2$-free graphs.  Finally, for whatever sufficiently large constant we will later choose $C$ to be, we set $\gamma := \min \{ \kappa / K, \pi/ (2 v(H_1)^{v(H_1)}) \}$, where $K = K(C, H_1, H_2)$ is the constant from Claim~\ref{clm:viableforbid}.

\medskip
 
Now let $\mc E_0$ be the event that there are more than $\kappa n^{v(H_2) - e(H_2) / m_2(H_2)}$ viable forbidden copies of $H_2$, which, by Claim~\ref{clm:viableforbid} and the fact that $\kappa \geq \gamma K$, we may assume does not occur.  Since all copies of $H_2$ in $B$ must be viable and forbidden, it follows that $B$ has at most $\kappa n^{v(H_2) - e(H_2) / m_2(H_2)}$ copies of $H_2$.  By Corollary~\ref{cor:Hcontainers}, there is then some $T \in \mc T \cap \mc P(E(B))^q$ with $B \subseteq \psi(T)$.  Let $\mc E_2(T)$ be the event that the set of edges $E(T) = \cup_{j=1}^q T^{(j)}$ is viable, which must hold for them to be in $B$.  By Claim~\ref{clm:viableprob}, we have
\begin{equation} \label{eqn:E2bound}
\mathbb{P}(\mc E_2(T)) \le \rho^{\card{E(T)}}.
\end{equation}

Next define $R = R(T,L) := E(K_n) \setminus \left( \psi(T) \cup L \right)$, and consider the $3$-colouring of $E(K_n)$ where all edges in $R$ receive colour $1$, edges in $\psi(T)$ have colour $2$ and those in $L \setminus \psi(T)$ are coloured $3$.  Corollary~\ref{cor:Hcontainers} guarantees that there are at most $\pi \binom{n}{v(H_2)} / 2$ copies of $H_2$ coloured $2$, while Claim~\ref{clm:fewnoniso} guarantees that there are fewer than $\pi \binom{n}{2}$ copies of $K_2$ in colour $3$.  By Proposition~\ref{prop:supersatramsey}, we must therefore have at least $\pi \binom{n}{v(H_1)}$ copies of $H_1$ in $R$.  Let $\mc H$ be those copies of $H_1$ in $R$ that are not forbidden, and note that since $\card{\mc F_1} \le \gamma n^{v(H_1)}$, we have $\card{\mc H} \ge \pi \binom{n}{v(H_1)} / 2 \ge \xi n^{v(H_1)}$ for $\xi := \pi / (2 v(H_1)^{v(H_1)})$.

Let $\mc E_3(T,L)$ be the event that none of the copies of $H_1$ in $\mc H$ appear in $G(n,p)$, and note that this must hold.  Otherwise, the copy of $H_1 \in \mc H$, which is not forbidden, must be monochromatic red, since all blue edges of $G(n,p)$ lie either in $L$ or in $B \subseteq \psi(T)$, contradicting our assumption that the only red copies of $H_1$ are forbidden.

By Theorem~\ref{thm:janson}, we have $\mathbb{P}(\mc E_3(T,L)) \le \exp( - \xi \mu_1(H_1) / (2^{v(H_1) + 1} v(H_1)! ) )$.  As we saw in~\eqref{eqn:mu0calc}, since $H_1$ is strictly balanced with respect to $m_2(\cdot, H_2)$, it follows that $\mu_1(H_1) = n^{v(H_1)}p^{e(H_1)} = C^{e(H_1)} n^{2 - 1/m_2(H_2)}$.  Thus, setting $\xi' := \xi / (2^{v(H_1)+1} v(H_1)!)$, we have
\begin{equation} \label{eqn:E3bound}
\mathbb{P}(\mc E_3(T,L)) \le \exp( - \xi C^{e(H_1)} n^{2 - 1/m_2(H_2)} / (2^{v(H_1) + 1} v(H_1)! ) ) = \exp(- \xi' C^{e(H_1)} n^{2 - 1/m_2(H_2)}).
\end{equation}

\medskip

To summarise, in order for $G(n,p)$ to have a $2$-colouring where all the red copies of $H_1$ and blue copies of $H_2$ are forbidden, we either need $\mc E_0$ or $\mc E_1$ to hold, or, if they do not, for $\mc E_2(T) \wedge \mc E_3(T,L)$ to hold for some choice of $T \in \mc T$ and some specified set $L$ of edges in non-isolated copies of $H_1$.  We can thus estimate the probability that $G(n,p)$ is not robustly Ramsey by taking a union bound over these results:
\begin{equation} \label{eqn:errorboundb}
	\mathbb{P}( G(n,p) \textrm{ is not robustly } (H_1, H_2)\textrm{-Ramsey}) \le \mathbb{P}(\mc E_0) + \mathbb{P}(\mc E_1) + \sum_{T, L} \mathbb{P}(\mc E_2(T) \wedge \mc E_3(T,L)).
\end{equation}

Since $\mc E_2(T)$ --- that all of the edges from $E(T)$ appear in $B$ --- is a monotone increasing graph property, while $\mc E_3(T,L)$ --- that none of the copies of $H_1$ from $R(T,L)$ appear in $G(n,p)$ --- is a monotone decreasing graph property, it follows from the FKG Inequality (see~\cite[Theorem 6.3.3]{as}) that $\mathbb{P}(\mc E_2(T) \wedge \mc E_3(T,L)) \le \mathbb{P}(\mc E_2(T)) \mathbb{P}(\mc E_3(T,L))$.  Using the upper bounds from~\eqref{eqn:E2bound} and~\eqref{eqn:E3bound}, we can bound the sum by
\[ \sum_{T,L} \mathbb{P}\left( \mc E_2(T) \wedge \mc E_3(T,L) \right) \le \exp\left( - \xi' C^{e(H_1)} n^{2 - 1/m_2(H_2)} \right) \sum_L \sum_T \rho^{\card{E(T)}}. \]

There are at most $n^{2 - 1/m_2(H_2) - \delta}$ non-isolated copies of $H_1$ (as otherwise we are covered by $\mc E_1$), and hence the sum $L$ only runs over sets of at most $e(H_1) n^{2 - 1/m_2(H_2) - \delta}$ edges.  Thus
\begin{align} \label{eqn:summingL}
	\sum_{T,L} \mathbb{P} \left( \mc E_2(T) \wedge \mc E_3(T,L) \right) &\le \exp \left( - \xi' C^{e(H_1)} n^{2- 1/m_2(H_2)} \right) \left( \sum_{\ell = 0}^{e(H_1) n^{2 - 1/m_2(H_2) - \delta}} \binom{\binom{n}{2}}{\ell} \right) \sum_T \rho^{\card{E(T)}} \notag \\
	&\le \exp\left( - \xi' C^{e(H_1)} n^{2 - 1/m_2(H_2)} + o \left( n^{2 - 1/m_2(H_2)} \right) \right) \sum_T \rho^{\card{E(T)}}.
\end{align}

We now bound the sum over $T$ as we did in Case (a).  From Corollary~\ref{cor:Hcontainers}, we know $s = \card{E(T)} \le M_n := q \beta n^{2 - 1/m_2(H_2)}$.  We group the $q$-tuples $T$ by the size of $E(T)$, observing that each set of edges $E(T)$ corresponds to at most $2^{q\card{E(T)}}$ $q$-tuples in $\mc T$.  Recalling the definition of $\rho$ from Claim~\ref{clm:viableprob}, we have
\[ \sum_T \rho^{\card{E(T)}} \le \sum_{s = 0}^{M_n} \sum_{\substack{T \in \mc T: \\ \card{E(T)} = s}} \rho^s \le \sum_{s=0}^{M_n} \binom{n^2}{s} \left( 2^q \rho \right)^s \le \sum_{s=0}^{M_n} \left( \frac{v(H_1)^2  e 2^q C^{e(H_1)} n^{2 - 1/m_2(H_2)} }{s} \right)^s. \]
As before, this summand is maximised when $s = M_n$, yielding, for $\lambda := v(H_1)^2 e 2^q / (q \beta)$,
\[ \sum_T \rho^{\card{E(T)}} \le \left( q \beta n^{2 - 1/m_2(H_2)} + 1 \right) \left( \lambda C^{e(H_1)} \right)^{q \beta n^{2 - 1/m_2(H_2)}}. \]
Making the substitution in~\eqref{eqn:summingL}, we find
\begin{equation} \label{eqn:E2E3sum}
\sum_{T,L} \mathbb{P}\left( \mc E_2(T) \wedge \mc E_3(T,L) \right) \le \exp \left( \left( - \xi' C^{e(H_1)} + q \beta \ln \left( \lambda C^{e(H_1)} \right)+ o(1) \right) n^{2 - 1/m_2(H_2)} \right).
\end{equation}

Thus, if we choose $C$ to be a sufficiently large constant (with respect to $\xi', q, \beta$ and $\lambda$), this sum is $o(1)$.  By Claims~\ref{clm:viableforbid} and~\ref{clm:fewnoniso} respectively, so are $\mathbb{P}(\mc E_0)$ and $\mathbb{P}(\mc E_1)$.  Hence, by~\eqref{eqn:errorboundb}, it follows that $G(n,p)$ is with high probability robustly $(H_1, H_2)$-Ramsey with respect to $(\mc F_1, \mc F_2)$. \qed

To complete the proof, we prove Claim~\ref{clm:viableforbid}.

\begin{proof}[Proof of Claim~\ref{clm:viableforbid}]
In order for a forbidden copy $H \in \mc F_2$ of $H_2$ to be viable, it must have a witness --- that is, a collection of pairwise edge-disjoint copies $W_e$ of $H_1$ containing each edge $e \in E(H)$.  We shall bound the number of witnesses for forbidden copies of $H_2$, thereby obtaining an upper bound on the number of viable forbidden copies of $H_2$.

Note that there are various non-isomorphic types of witnesses, depending on how the copies $W_e$ of $H_1$ are attached to the edges of the forbidden copy of $H_2$, and whether or not they share vertices.  Let $\Gamma_1, \hdots, \Gamma_m$ represent the different isomorphism classes, where $m := m(H_1, H_2)$ is some constant.  Although these witnesses can have different numbers of vertices, up to a maximum of $v(H_2) + e(H_2)(v(H_1) - 2)$, the edge-disjointness of the $W_e$ ensures they each have exactly $e(H_1) e(H_2)$ edges.

We will show that for $p = Cn^{-1/m_2(H_1,H_2)}$, if we set $K' := 2 \gamma C^{e(H_1) e(H_2)}$, then for each $i \in [m]$ the probability that there are at most $\gamma K' n^{v(H_2) - e(H_2) / m_2(H_2)}$ copies of $\Gamma_i$ with the central copy of $H_2$ coming from $\mc F_2$ is $O(1/n)$.  Then, taking $K := m K'$, a union bound over the different types of witnesses $\Gamma_i$ shows that the probability of there being more than $\gamma K n^{v(H_2) - e(H_2)/m_2(H_2)}$ viable forbidden copies of $H_2$ is also $O(1/n)$.

Let us fix some witness $\Gamma_i$, where $i \in [m]$, and let $X_i$ denote the number of copies of $\Gamma_i$ in $G(n,p)$ with the central copy of $H_2$ coming from $\mc F_2$.  If we first select this central copy, and then the remaining vertices of $\Gamma_i$, we see there are at most $\card{\mc F_2} n^{v(\Gamma_i) - v(H_2)} = \gamma n^{v(\Gamma_i)}$ possible witnesses of this form, each appearing with probability $p^{e(H_1)e(H_2)}$.  Thus $\mathbb{E}[X_i] \le \gamma n^{v(\Gamma_i)} p^{e(H_1) e(H_2)}$.

We now differentiate between two cases.

\subsubsection*{Case I: $v(\Gamma_i) \le v(H_2) + e(H_2)(v(H_1) - 2) - 1$.}

In this case we have 
\[ \mathbb{E}[X_i] \le \gamma n^{v(H_2) + e(H_2)(v(H_1) - 2) - 1} p^{e(H_1)e(H_2)} = \gamma n^{v(H_2) - 1} \left( n^{v(H_1) - 2} p^{e(H_1)} \right)^{e(H_2)}.\]
Since $n^{v(H_1) - 2} p^{e(H_1)} = C^{e(H_1)} n^{-1/m_2(H_2)}$, we have $\mathbb{E}[X_i] \le \gamma C^{e(H_1) e(H_2)} n^{v(H_2) - e(H_2)/m_2(H_2) - 1}$.  By Markov's inequality, $\mathbb{P}(X_i > \gamma K' n^{v(H_2) - e(H_2)/m_2(H_2)}) = O(1/n)$, as desired.

\subsubsection*{Case II: $v(\Gamma_i) = v(H_2) + e(H_2)(v(H_1) - 2)$.}

In this case the copies $\{ W_e : e \in E(H_2) \}$ of $H_1$ in $\Gamma_i$ are all vertex-disjoint except for where a pair  intersect in precisely one vertex in the central copy of $H_2$.  

The calculation above shows $\mathbb{E}[X_i] \le \gamma C^{e(H_1)e(H_2)} n^{v(H_2) - e(H_2) / m_2(H_2)}$.
We can again use Lemma~\ref{lem:covariance} to bound $\Var(X_i)$, applying the lemma with $H := \Gamma_i$ and $\mc H$ the collection of all possible witnesses of this type for forbidden copies of $H_2$.  As in the proof of Claim~\ref{clm:fewforbidden}, this gives
\begin{align*}
\Var(X_i) \le \Delta + \mathbb{E}[X_i] &\le \frac{\gamma 2^{v(\Gamma_i)} v(\Gamma_i)! (n^{v(\Gamma_i)} p^{e(\Gamma_i)})^2}{\mu_0(\Gamma_i)} + \gamma n^{v(\Gamma_i)} p^{e(\Gamma_i)} \\
	&= O\left( \frac{\left( n^{v(\Gamma_i)} p^{e(\Gamma_i)} \right)^2}{\mu_1(\Gamma_i)} \right) = O\left( \frac{\left(n^{v(H_2) - e(H_2)/m_2(H_2)}\right)^2}{\mu_1(\Gamma_i)}\right),
\end{align*}
where $\mu_1(\Gamma_i)$ is as defined in Theorem~\ref{thm:janson}.  The penultimate equality above then follows from the fact that $\mu_1(\Gamma_i) = \min \{ \mu_0(\Gamma_i), n^{v(\Gamma_i)} p^{e(\Gamma_i)} \}$.
We shall shortly show that $\mu_1(\Gamma_i) = \Omega(n)$.  Chebyshev's inequality then gives the desired bound:
\begin{align*} \mathbb{P} \left( X_i > K' n^{v(H_2) - e(H_2)/m_2(H_2)} \right) &\le \mathbb{P} \left( X_i - \mathbb{E}[X_i] > \gamma C^{e(H_1) e(H_2)} n^{v(H_2) - e(H_2)/m_2(H_2)} \right) \\
	&\le \frac{\Var(X_i)}{\left(\gamma C^{e(H_1)e(H_2)} n^{v(H_2) - e(H_2)/m_2(H_2)} \right)^2} = O \left( \frac{1}{n} \right).
\end{align*}

To finish, we establish a lower bound for $\mu_1(\Gamma_i) = \min \{ n^{v(F)} p^{e(F)} : F \subseteq \Gamma_i, e(F) \ge 1 \}$, and let $F_0$ be the subgraph of $\Gamma_i$ that minimises this expression.  Recall that for each edge $h$ in the central forbidden copy $H \in \mc F_2$ of $H_2$, we have a copy $W_h$ of $H_1$ containing $h$.  Let $F_0(h)$ be the subgraph of $H_1$ induced by $F_0$ on $W_h$, and let $w_h := e(F_0(h))$ be the number of edges it contains.  We then have $e(F_0) = \sum_{h \in H_2} w_h$.

When we count the vertices of $F_0$, we observe that the vertices from the central $H_2$ will belong to several of the $F_0(h)$.  Thus, for each $h \in E(H_2)$, we let $v_h := v(F_0(h))$, and let $r_h \in \{0,1,2\}$ denote the number of the vertices from $h$ that are in $F_0$.  Let $F_0^*$ denote the subgraph of the central $H_2$ contained within $F_0$; in particular, we have $r_h = 2$ if and only if $h \in E(F_0^*)$.  Moreover, $v(F_0) = v(F_0^*) + \sum_{h \in E(H_2)} (v_h - r_h)$.

Thus we have 
\[ \mu_1(\Gamma_i) = n^{v(F_0)} p^{e(F_0)} = n^{v(F_0^*)} \prod_{h \in E(H_2)} n^{v_h - r_h} p^{w_h} \ge  C^{e(F_0)} n^{v(F_0^*)} \prod_{\substack{h \in E(H_2) : \\ w_h \ge 1}} n^{v_h - r_h - w_h/m_2(H_1, H_2)}. \]
For each $h$ in the final product, since $F_0(h) \subseteq H_1$, the definition of $m_2(H_1, H_2)$ implies $w_h \le (v_h - 2 + 1/m_2(H_2))m_2(H_1, H_2)$.  Therefore
\[ \mu_1(\Gamma_i) \ge C^{e(F_0)} n^{v(F_0^*)} \prod_{\substack{h \in E(H_2): \\ w_h \ge 1}} n^{2 - r_h - 1/m_2(H_2)} \ge C^{e(F_0)} n^{v(F_0^*) + \card{\{ h \in E(H_2): w_h \ge 1, r_h = 0 \}} - e(F_0^*)/m_2(H_2)}, \]
where the last inequality is due to the fact that, since $m_2(H_2) \ge 1$, the exponent $2 - r_h - 1/m_2(H_2)$ is at least $1$ if $r_h = 0$, is non-negative when $r_h = 1$, and is $-1/m_2(H_2)$ when $r_h = 2$ or, equivalently, $h \in E(F_0^*)$.

If $e(F_0^*) \ge 1$ (and therefore $v(F_0^*) \ge 2$), then by definition of $m_2(H_2)$ we have $e(F_0^*) \le (v(F_0^*) - 2) m_2(H_2) + 1$. Substituting this into the exponent, we find $n^{v(F_0^*) - e(F_0^*)/m_2(H_2)} \ge n^{2 - 1/m_2(H_2)} \ge n$, and hence $\mu_1(\Gamma_i) = \Omega(n)$.

Therefore we may assume $e(F_0^*) = 0$, in which case $\mu_1(\Gamma_i) = \Omega(n^{v(F_0^*) + \card{\{ h \in E(H_2): w_h \ge 1, r_h = 0\}}})$.  This exponent is always at least $1$, since either $v(F_0^*) \ge 1$, or $r_h = 0$ for all $h \in E(H_2)$ and, since $e(F_0) \ge 1$, we must have $w_h \ge 1$ for some $h$ as well.  Thus we indeed have $\mu_1(\Gamma_i) = \Omega(n)$, completing the proof of this claim.
\end{proof}

\subsubsection{Random graphs are globally Ramsey}

Having proven part (i) of Theorem~\ref{thm:global1ramsey}, we finally  deduce the global Ramsey properties of $G(n,p)$ from the above proofs, thereby establishing the remaining parts (ii) and (iii).  Here we wish to show that, for appropriate probabilities, $G(n,p)$ is $\mu$-globally $(H_1, H_2)$-Ramsey, for $\mu > 0$ constant and $\mu = n^{-\beta}$ respectively.

\medskip

We begin with the assumption (a), that $m_2(H_1) = m_2(H_2)$.  For (ii), let $\mu > 0$ be a fixed constant, and set $C_2 := C_1 \mu^{-1/m_2(H_1, H_2)}$, where $C_1$ is the constant from part (i).  To show that $G(n,p)$ is $\mu$-globally $(H_1, H_2)$-Ramsey when $p \ge C_2 n^{-1/m_2(H_1, H_2)}$, it will suffice to apply the union bound over all subsets $U \subseteq [n]$ of $n_0 := \mu n$ vertices.

Indeed, for any such set $U$, we have $G(n,p)[U] \sim G(n_0,p)$, where $p = C_2 n^{-1/m_2(H_1, H_2)} = C_1 n_0^{-1/m_2(H_1, H_2)}$.  By~\eqref{eqn:errorbounda}, the probability that $G(n,p)[U]$ is not $(H_1, H_2)$-Ramsey is at most $\exp \left( - C_1 \pi n_0^{2 - 1/m_2(H_1)} / 8 \right) + o(1)$, where $\pi = \pi(H_1, H_2, K_2)$ from Proposition~\ref{prop:supersatramsey}.

However, the $o(1)$ error term came from the event $\mc E_0$, where many forbidden copies of $H_i$ appeared in $G(n_0,p)$.  Since we do not have forbidden copies in our current setting, we do not incur this error.  Thus, recalling that $m_2(H_1) \ge 1$, and assuming, as we may freely do, $C_1 \ge 8 \pi \mu^{-1}$, we see that the probability that $G(n,p)[U]$ is not $(H_1, H_2)$-Ramsey is at most $\exp(-n)$.  We may therefore take a union bound over all such sets $U$, of which there are fewer than $2^n$, which shows  that $G(n,p)$ is $\mu$-globally $(H_1, H_2)$-Ramsey with high probability.

\medskip

For part (iii), we instead set $n_0 := n^{1 - \beta}$, and set $C_3 := C_1$.  Letting $U \subseteq [n]$ be a set of $n_0$ vertices, we have $G(n,p)[U] \sim G(n_0, p)$, where $p \ge C_1 n_0^{-1/m_2(H_1, H_2)}$.  As above, the probability that $G(n,p)[U]$ is not $(H_1, H_2)$-Ramsey is at most $\exp( - C_1 \pi n_0^{2 - 1/m_2(H_1)} / 8)$.

We can now take a union bound over all such sets $U$, of which there are $\binom{n}{n_0} \le \exp \left( n_0 \ln \left( n e / n_0 \right) \right) = \exp ( n_0^{1 + o(1)} )$.  Hence the probability that there is a non-Ramsey subgraph induced on some $U$ is at most $\exp( n_0^{1 + o(1)} - C_1 \pi n_0^{2 - 1/m_2(H_1)} / 8)$.  Since $m_2(H_1) > 1$, it follows that $G(n,p)$ is $n^{-\beta}$-globally $(H_1, H_2)$-Ramsey with high probability.

\medskip

The argument under the assumption (b), that $H_1$ is strictly balanced with respect to $m_2(\cdot, H_2)$, is very similar.  We will again run a union bound over all sets $U$ of size $n_0 := \mu n$
(for constant $\mu >0$ and $\mu=n^{-\beta}$), using the fact that $G(n,p)[U] \sim G(n_0, p)$ and choosing the constants $C_2$ and $C_3$ such that $p \ge C_1 n_0^{-1/m_2(H_1, H_2)}$.

In~\eqref{eqn:errorboundb} we bounded from above the probability of $G(n,p)$ not being robustly $(H_1, H_2)$-Ramsey.  Since we do not have collections of forbidden copies of $H_1$ and $H_2$, we can again omit the error term from the event $\mc E_0$ of there being too many viable forbidden copies of $H_2$.  We then have
\[ \mathbb{P}(G(n,p)[U] \textrm{ is not } (H_1, H_2)\textrm{-Ramsey} ) \le \mathbb{P}(\mc E_1) + \sum_{T,L} \mathbb{P}( \mc E_2 (T) \wedge \mc E_3(T,L) ). \]

Recall that $\mc E_1$ is the event that we have more than $n^{2 - 1/m_2(H_2) - \delta}$ non-isolated copies of $H_1$.  In Claim~\ref{clm:fewnoniso} we only obtained polynomially small bounds on its probability, and so we cannot afford to take a union bound over this event holding for each induced subgraph $G(n,p)[U]$.  Instead, we replace the event by $\mc E_1'$, the event that $G(n,p)$ has more than $n_0^{2 - 1/m_2(H_2) - \delta}$ non-isolated copies of $H_1$.  Since a non-isolated copy of $H_1$ in $G(n,p)[U]$ must be non-isolated in $G(n,p)$ as well, this modified event would suffice for our purposes, and we can avoid the union bound for $\mc E_1$.

To see that $\mc E_1'$ holds with vanishingly small probability, observe that in Claim~\ref{clm:fewnoniso}, the expected number of pairs of intersecting copies of $H_1$, which we bounded by $O \left( n^{2 - 1/m_2(H_1,H_2) - 2 \delta} \right)$, was polynomial in $n$ and $p$.  Hence when we increase $p$ to either $C_2 n^{-1/m_2(H_1, H_2)}$ or $C_3 n^{-(1- \beta)/m_2(H_1, H_2)}$, provided $\beta < \beta_0(H_1, H_2)$, this expected number will grow to at most $O \left( n^{2 - 1/m_2(H_1, H_2) - 5 \delta / 3} \right)$.  Similarly, since $n_0 \ge n^{1 - \beta_0}$, we have $n_0^{2 - 1/m_2(H_2) - \delta} = \Omega \left( n^{2 - 1/m_2(H_2) - 4 \delta / 3} \right)$.  Thus we can still apply Markov's inequality to deduce $\mathbb{P}(\mc E_1') = o(1)$.

\medskip

This leaves us with the sum $\sum_{T,L} \mathbb{P}\left( \mc E_2(T) \wedge \mc E_3(T,L) \right)$, which we saw in~\eqref{eqn:E2E3sum} can be bounded by $\exp(- \Omega\left(n^{2 - 1/m_2(H_2)}\right)$.  As under the assumption (a), we can afford to take a union bound over all $\binom{n}{n_0}$ choices for the set $U$, and still have the probability that there is some induced subgraph $G(n,p)[U]$ where $\mc E_2(T) \wedge \mc E_3(T,L)$ holds be $o(1)$.  Hence it follows that $G(n,p)$ is with high probability $\mu$-globally $(H_1, H_2)$-Ramsey in this case as well, completing the proof of Theorem~\ref{thm:global1ramsey}. \qed

\end{document}